\documentclass[11pt]{article}
\usepackage{amsmath,amssymb,graphicx}
\usepackage[makeroom]{cancel}
\usepackage{lscape}
\pdfoutput=1

\newtheorem{theo}{Theorem}

\newtheorem{lemm}[theo]{Lemma}

\newtheorem{defn}{Definition}

\newenvironment{proof}{\medskip {\noindent\bf Proof:}}{$\parallel$ \medskip}
\newenvironment{rema}{{\noindent\bf Remark:}}{$\parallel$ \smallskip}

\begin{document}

\title{Median Confidence Regions in a Nonparametric Model}

\author{Edsel A.\ Pe\~na\footnote{E.\ Pe\~na is Professor, Department of Statistics, University of South Carolina,
Columbia, SC 29208 USA. {\em E-Mail:} pena@stat.sc.edu}
\and Taeho Kim\footnote{T.\ Kim is a PhD Student, Department of Statistics, University of South Carolina, Columbia, SC 29208 USA. {\em E-Mail:}\ taeho@email.sc.edu}
}

\date{\today}

\maketitle

\begin{abstract}
The nonparametric measurement error model (NMEM) has
$X_i = \Delta + \epsilon_i, i=1,2,\ldots,n; \Delta\in\Re; \epsilon_i\ \mbox{IID}\ F(\cdot) \in \mathfrak{F}_{c,0},$
where $\mathfrak{F}_{c,0}$ is the class of all continuous distributions with median $0$, so $\Delta$ is the median parameter of $X$. This paper deals with the problem of constructing a confidence region (CR) for $\Delta$ under the NMEM. This problem arises in many settings, including inference about the median lifetime of a complex system arising in engineering, reliability, biomedical, and public health settings. Current methods of constructing CRs for $\Delta$ are discussed, including the $T$-statistic based CR and the Wilcoxon signed-rank statistic based CR, arguably the two default methods in applied work when a confidence interval about the center of a distribution is desired. Optimal equivariant CRs are developed with focus on subclasses of $\mathfrak{F}_{c,0}$ when looking at their expected Lebesgue content. Applications to a real car mileage efficiency data set and Proschan's air-conditioning data set are demonstrated. Simulation studies to compare the performances of the different CR methods were undertaken.  Results of these studies indicate that the sign-statistic based CR and the optimal CR focused on symmetric distributions satisfy the confidence level requirement, though they tended to have higher contents; while two of the bootstrap-based CR procedures and one of the developed adaptive CR tended to be a tad more liberal but with smaller contents. A critical recommendation is that, {\em under the NMEM}, both the $T$-statistic based and Wilcoxon signed-rank statistic based confidence regions should {\bf not} be used since they have degraded confidence levels and/or inflated contents.

\smallskip


\smallskip


\end{abstract}

\section{Introduction and Motivation}
\label{sect-intro}

Given a univariate distribution function $G$, the two most common measures of central tendency are the mean $\mu = \int x G(dx)$, provided it exists ($\int |x| G(dx) < \infty$), and the median $\Delta = \inf\{x \in \Re:\ G(x) \ge 1/2\}$. The mean need not always exist, whereas the median always exists. 
Under symmetric distributions, and when the mean exists, then the mean and the median coincide. This paper is concerned with making statistical inferences about the median $\Delta$ of a distribution.
A popular model leading to the problem of making inference about the median of a distribution is the so-called measurement error model. In this model $\Delta$ represents a quantity of interest which is unknown, and when one measures its value, the observed value $x$ is a realization of the random variable
\begin{equation}
\label{measurement error model}
X = \Delta + \epsilon,
\end{equation}
where $\epsilon$ represents a measurement error with a continuous distribution $F(\epsilon)$ whose median equals zero. As such, the distribution of $X$ is $G(x) = F(x-\Delta)$. Typically, $F(\cdot)$ is assumed to be a zero-mean normal distribution, but this assumption is not tenable in many situations. For instance, in dealing with event times in biomedical, reliability, engineering, economic, and social settings, the error distribution need not even be symmetric. This is also the case when dealing with economic indicators such as per capita income, retirement savings, etc. As such, a general model is to simply assume that the error distribution $F$ belongs to the class of all continuous distributions with medians equal to zero. This class will be denoted by $\mathfrak{F}_{c,0}$. 

Another situation where this problem arises is when dealing with a complex engineering system, such as the motherboard of a laptop computer or some technologically-advanced car (e.g., a Tesla Model S sedan). Such a system will be composed of many different components configured according to some structure function, with the components having different failure-time distributions and some of them possibly acting dependently on each other. Of main interest for such a system will be its time-to-failure (also called lifetime) denoted by $X$. Because of the complexity of the system, it may not be feasible to analyze the distribution of $X$ by taking into account each of the failure time distributions of the components and the system's structure function which represents the configuration of the components to form the system. Thus a simplified and practically feasible viewpoint is to assume that the system's life distribution is some continuous distribution $G$. One may then be interested in the median $\Delta$ of this distribution $G$. 

Thus, in these situations, the observable random variable $X$ is assumed to have a distribution $F(x-\Delta)$ with $F(\cdot) \in \mathfrak{F}_{c,0}$ and $\Delta \in \Re$ being the median of $X$. This will be referred to as the one-population nonparametric measurement error model, abbreviated NMEM. This is the simplest among the measurement error models. The goal is to infer about the parameter of interest $\Delta$ with $F(\cdot)$ acting as an infinite-dimensional nuisance parameter. We shall be interested in this paper in the construction of a confidence region (CR) for $\Delta$ based on a random sample of observations of $X$. This definitely is a classic problem since the construction of a confidence interval for the median was even discussed in (\cite{Tho36}). More generally, quantiles instead of just the median may be of interest, and the methods developed here could be adaptable to inference about general quantiles.

Arguably, confidence regions for a parameter are preferable than point estimates since they address simultaneously the issue of how close to the truth (measured through the content of the region) and how sure about such closeness to the truth (measured by the confidence region coefficient). For more discussions on desirability of confidence regions see, for instance, the introduction in \cite{CicEfr96} and chapter 5 in \cite{DavHin97}.  Of course, one could typically accompany a point estimate (PE) by an estimate of its standard error (ESE), but then the user still needs to deduce closeness and sureness based on the PE and the ESE, usually a non-trivial matter if to be done {\em properly}.

We introduce some notations and definitions. Let $X_1, X_2, \ldots, X_n$ be independent and identically distributed (IID) random variables (a random sample) from $F(x-\Delta)$, where $\Delta \in \Re$ and $F \in \mathfrak{F}_{c,0}$. The mathematical problem is to construct a confidence region (CR) for the parameter $\theta(F,\Delta) = \Delta$ with $F$ an infinite-dimensional nuisance parameter. Denote by $\mathfrak{X}$ the range space of $\mathbf{X} = (X_1,\ldots,X_n)$ which will be endowed with a $\sigma$-field $\mathcal{X}$. We also denote by $\mathfrak{B}$ the Borel $\sigma$-field of $\Re$, and this will be endowed with the $\sigma$-field of subsets of $\mathfrak{B}$ consisting of its countable and co-countable subsets, with this $\sigma$-field denoted by $\mathcal{B}$.

\begin{defn}
\label{defn-CR}
Fix an $\alpha \in (0,1)$. Let $\mathbf{X} = (X_1,\ldots,X_n) \in \mathfrak{X}$ be IID from $F(x-\Delta)$. A measurable mapping $\Gamma: (\mathfrak{X},\mathcal{X}) \rightarrow (\mathfrak{B},\mathcal{B})$ is called a $100(1-\alpha)\%$ region estimator or confidence region (CR) for $\Delta$ if
$P_{(F(\cdot),\Delta)}\{\Delta \in \Gamma(\mathbf{X})\} \ge 1 - \alpha$
for every $(F(\cdot),\Delta) \in \mathfrak{F}_{c,0} \times \Re$.
\end{defn}

\begin{rema}
In later developments, we will allow the CR $\Gamma$ to also depend on a randomizer $U$, a standard uniform random variable independent of $\mathbf{X}$. This is to be able to achieve exactly the desired confidence level $1-\alpha$. In such a case, $\Gamma: \mathfrak{X} \times [0,1] \rightarrow \mathfrak{B}$ and $\Gamma(\mathbf{x},u)$ will be the realized CR when $\mathbf{X} = \mathbf{x}$ and $U = u$. However, even if we allow for randomized CRs, we will usually suppress writing the $U$ in $\Gamma(\mathbf{X},U)$ and simply write $\Gamma(\mathbf{X})$.
\end{rema}

Aside from satisfying the desired confidence coefficient in Definition \ref{defn-CR}, the quality of a CR depends on some measure of its content. Let $\nu(\cdot)$ be Lebesgue measure on $(\Re,\mathfrak{B})$. We will measure the content of a CR $\Gamma$ for $\Delta$ via
\begin{equation}
\label{content of CR}
\mathcal{C}[\Gamma;(F(\cdot),\Delta)] = E_{(F(\cdot),\Delta)} \nu[\Gamma(\mathbf{X})].
\end{equation}

In Definition \ref{good CRs} below we have the notion of uniformly best CRs. Our goal is to determine those CRs for $\Delta$ that possess such optimality properties.

\begin{defn}
\label{good CRs}
Let $\bar{\mathfrak{F}}_{c,0}$ be a subclass of $\mathfrak{F}_{c,0}$. A $100(1-\alpha)\%$ CR $\Gamma^*$ for $\Delta$ is a uniformly best CR for $\Delta$ under the subclass $\bar{\mathfrak{F}}_{c,0}$ if for any other  $100(1-\alpha)\%$ CR $\Gamma$,
$$\mathcal{C}[\Gamma^*;(F(\cdot),\Delta)] \le \mathcal{C}[\Gamma;(F(\cdot),\Delta)]$$
for all $(F(\cdot),\Delta) \in \bar{\mathfrak{F}}_{c,0} \times \Re$. If $\bar{\mathfrak{F}}_{c,0} = \mathfrak{F}_{c,0}$, then $\Gamma^*$ will be said to be the uniformly best CR for $\Delta$.
\end{defn}

The major contribution of this work is the development of $100(1-\alpha)\%$ randomized region estimators or confidence regions, possibly approximate, for the median $\Delta$ under the NMEM, of form
\begin{eqnarray*}
\Gamma(X,U) & = & \left[\bigcup_{\left[\{k \in \{0,1,\ldots,n\}:\ b(k;n,1/2) > c^* \hat{l}(k)\}\right] } \left[X_{(k)}, X_{(k+1)}\right)\right] \bigcup \\
&&\{U \le \gamma\}  \left[\bigcup_{\{k \in \{0,1,\ldots,n\}:\ b(k;n,1/2) = c^* \hat{l}(k)\}} \left[X_{(k)}, X_{(k+1)}\right) \right],
\end{eqnarray*}
where $b(k;n,1/2) = {n \choose k} 2^{-n}$ and $\hat{l}(k)$ is an appropriate estimator of $l(k;F) = E\{X_{(k+1)}\} - E\{X_{(k)}\}$. The randomizer $U$ is a uniform random variable, while $c^*$ is the infimum over all $c \in \Re$ satisfying $P\{b(B;n,1/2) > c \hat{l}(B)\} \le 1 - \alpha$ where $B$ is a binomial random variable with parameters $n$ and $1/2$. A specific form of $\hat{l}(k)$ that leads to a reasonable CR is given by
\begin{displaymath}
\hat{l}(k) = {n \choose k} \int_{-\infty}^{\infty} \hat{F}(w)^{k} [1 - \hat{F}(w)]^{n-k} dw
\end{displaymath}
where $\hat{F}(w) = \sum_{i=1}^n I\{X_i - \hat{\Delta} \le w\}/n$, the empirical distribution function of $X_i - \hat{\Delta}, i=1,2,\ldots,n$, with $\hat{\Delta}$ being the sample median. The specific CR above will be developed in section \ref{sect-adaptive methods}.  Prior to the development of the specific CRs, in section \ref{sect-Optimal CR} we utilize invariance ideas to derive the general form of the almost-optimal equivariant CR for $\Delta$ under the NMEM but still under the assumption that $F$ is known. Then, we address the question of how to deal with the fact that $F$ is not actually known, leading to the region estimator above. Two other region estimators which are focused toward the class of symmetric distributions and the class of negative exponential distributions, but still valid for the general NMEM model, will be developed in sections \ref{sect-symmetric distributions} and \ref{sect-exponential distributions}, respectively. In the simulation studies, the procedure focused on symmetric distributions actually performed quite robustly under varied distributions (even for the non-symmetric distributions) in terms of coverage probability and it had mean content superior to the procedure based on the sign statistic. Prior to studying the performance of these new region estimators, we briefly describe existing (`off-the-shelf') region estimators for the median in section \ref{sec-current methods}. We then proceed to demonstrate these different region estimators by applying to two data sets in section \ref{sec-illustration}. Section \ref{sec-Simulation Comparisons of Some of the Methods} will present the results of simulation studies comparing the performances of these region estimators under different underlying distributions by examining their mean contents and their achieved confidence levels. These comparisons are also of major importance since they demonstrate CR procedures that should be preferred and which CRs should not be used under the NMEM . Section \ref{sec-conclusion} will provide some concluding remarks.

\section{Development of Optimal CRs}
\label{sect-Optimal CR}

\subsection{Invariant Models and Equivariant CRs}
\label{subsec-invariance}

We first review the notions of invariant statistical models and equivariant CRs (see, for instance, \cite{LehRom05}). We do this review in a more general framework than the concrete NMEM which is the focus of this paper. We note that sufficiency and invariance were major ideas utilized by Peter Hooper in several of his papers dealing with confidence sets and prediction sets, cf., \cite{Hoo82a,Hoo82b,Hoo86}.

Let $X$ be an observable random element taking values in a sample space $\mathfrak{X}$. The class of probability models governing $X$ is $\mathfrak{P}$ which consists of probability measures $P$'s on the measurable space $(\mathfrak{X},\mathcal{F})$, with $\mathcal{F}$ a suitable $\sigma$-field of subsets of $\mathfrak{X}$. Let $\tau: \mathcal{P} \rightarrow \mathfrak{T}$ be a functional, with $\tau(P)$ being the parameter of interest. A $100(1-\alpha)\%$ confidence region for $\tau(P)$ is a set-valued mapping $\Gamma: \mathfrak{X} \rightarrow \mathcal{T}$, where $\mathcal{T}$ is a class of subsets of $\mathfrak{T}$ such that
\begin{equation}
\label{general defn of CR}
P\{\tau(P) \in \Gamma(\mathbf{X})\} \ge 1 - \alpha, \ \forall P \in \mathfrak{P}.
\end{equation}

Let $G = \{g: \mathfrak{X} \rightarrow \mathfrak{X}\}$ be a family of transformations on $\mathfrak{X}$ that forms a group under an operation $\cdot$ and with identity element $1_G \equiv 1$. Let $\bar{G} = \{\bar{g}: \mathfrak{P} \rightarrow \mathfrak{P}\}$ be a group of transformations on $\mathfrak{P}$ such that there exists a homomorphism $\bar{h}: G \rightarrow \bar{G}$ and let $1 \equiv 1_{\bar{G}} = \bar{h}(1_G)$ be the identity element in $\bar{G}$. The statistical model is said to be $(G,\bar{G})$-invariant if
\begin{equation}
\label{invariance of model}
P\{gX \in A\} = \bar{g}P\{X \in A\}, \forall g \in G; A \in \mathcal{F}.
\end{equation}
In addition, let $\tilde{G} = \{\tilde{g}: \mathfrak{T} \rightarrow \mathfrak{T}\}$ be a group of transformations on $\mathfrak{T}$ such that there exists a homomorphism $\tilde{h}: G \rightarrow \tilde{G}$. The parametric functional $\tau(P)$ is said to be $(\bar{G},\tilde{G})$-equivariant if $\tau(\bar{g}P) = \tilde{g}\tau(P)$ for all $g \in G, P \in \mathfrak{P}$. 
Employing a decision-theoretic framework, define a loss function on $\mathfrak{T} \times  \sigma(\mathfrak{T})$ given by the $0-1$ loss function
$$L(\tau,C) = 1  - I\{\tau \in C\}.$$
We shall say the the loss function is $\tilde{G}$-invariant if $L(\tilde{g}\tau,\tilde{g}C) =  L(\tau,C)$ for every $g \in G$, $\tau \in \mathfrak{T}$, and $C \in \sigma(\mathfrak{T})$.  Given a confidence region $\Gamma(X)$, its risk function is
\begin{displaymath}
R(P,\Gamma) \equiv E_P\{L(\tau(P),\Gamma(X))\}  =  1 - P\{\tau(P) \in \Gamma(X)\}.
\end{displaymath}
As such, the condition for a $100(1-\alpha)\%$ confidence region $\Gamma(X)$ is equivalent to having $R(P,\Gamma) = E_P\{L(\tau(P),\Gamma(X))\} \le \alpha$ for every $P \in \mathfrak{P}$.
When a $(G,\bar{G})$-invariant statistical model is coupled with a $\tilde{G}$-invariant loss function, then we would say that the statistical problem of constructing a confidence region $\Gamma(\cdot)$ is $(G,\bar{G},\tilde{G})$-invariant.
A confidence region $\Gamma(X)$ is then said to be $(G,\tilde{G})$-equivariant if for every $g \in G$ and $x \in \mathfrak{X}$, we have that
$$\Gamma(gx) = \tilde{g}\Gamma(x) \equiv \{\tilde{g}t: t \in \Gamma(x)\}.$$
The Principle of Invariance then dictates that we should only utilize $(G,\tilde{G})$-equivariant confidence regions.

For an invariant confidence region problem, if $\Gamma(\cdot)$ is equivariant, then we have that, for every $g \in G$,
\begin{eqnarray*}
P\{\tau(P) \in \Gamma(X)\}   & = &   E_P\{1- L(\tau(P),\Gamma(X))\} 
  =   E_P\{1- L(\tilde{g}\tau(P),\tilde{g}\Gamma(X))\}  \\
& = & E_P\{1-L(\tau(\bar{g}P),\Gamma(gX)\} 
  =  E_{\bar{g}P} \{1-L(\tau(\bar{g}P),\Gamma(X))\} \\
&  = &   (\bar{g}P)\{\tau(\bar{g}P) \in \Gamma(X)\}.
\end{eqnarray*}
Furthermore, if the group $\bar{G}$ is transitive over $\mathfrak{P}$, meaning that for any given $P_0 \in \mathfrak{P}$ we have $\{\bar{g}P_0: \bar{g} \in G\} = \mathfrak{P}$, then it suffices to consider an arbitrary element $P_0 \in \mathfrak{P}$ to determine $P\{\tau(P) \in \Gamma(X)\}$ for all $P \in \mathfrak{P}$ since this equals the value using the arbitrary  $P_0$.

Recall that we also need to measure the quality of a confidence region by measuring its content using the quantity
$\mathcal{C}(P,\Gamma) = E_P[\nu(\Gamma(X))],$
where $\nu(\cdot)$ is a measure on $(\mathfrak{T},\sigma(\mathfrak{T}))$, e.g., Lebesgue measure. We seek those confidence regions with small $\mathcal{C}(P,\Gamma)$. Observe that for an equivariant $\Gamma(\cdot)$ in an invariant statistical model, we have for every $g \in G$ that
\begin{eqnarray*}
\mathcal{C}(P,\Gamma) & = & E_P\left[\nu(\Gamma(X))\right] 
 =  E_P\left[\nu\left(\tilde{g}^{-1}\Gamma(gX)\right)\right] 
 =   E_{\bar{g}P}\left[\nu\left(\tilde{g}^{-1}\Gamma(X)\right)\right].
\end{eqnarray*}
If it so happens that $\nu[\tilde{g}^{-1}\Gamma(x)] = \xi(g) \nu[\Gamma(x)]$ for all $x \in \mathfrak{X}$ and $g \in G$ and for some $\xi: G \rightarrow \Re$, then there is the possibility of finding a $\Gamma(\cdot)$ that satisfies the required confidence level and minimizes the content. We shall call this condition as quasi-invariance of $\nu$ with respect to $(G,\tilde{G})$. However, if $(G,\tilde{G})$-quasi-invariance of $\nu$ does not hold, then a uniformly best confidence region may not exist. But, a uniformly best confidence region on a {\em subfamily} $\mathfrak{P}_0 \subset \mathfrak{P}$ may still exist among the class of $100(1-\alpha)\%$ confidence regions over $\mathfrak{P}$. In \cite{Hoo82a,Hoo82b} quasi-invariance of the measure $\nu$ was imposed, but in some settings this may be unnatural such as in the NMEM under consideration in the current paper.

\subsection{Towards Optimal CRs for the Median}
\label{subsec-optimal median CR}

Consider now the problem of constructing a CR for the median $\Delta$ under the NMEM: 
$X_i = \Delta + \epsilon_i, i = 1,\ldots,n; \ \epsilon_i \stackrel{IID}{\sim} F(\cdot) \in \mathfrak{F}_{c,0}, \Delta \in \Re.$
Prior to invoking invariance, we first reduce via the Sufficiency Principle. Thus, we may assume that the observable random vector is $X_{()} = (X_{(1)},X_{(2)},\ldots,X_{(n)})$, the vector of order statistics which is a complete sufficient statistic. The appropriate sample space is therefore $\mathfrak{X} = \{(v_1,v_2,\ldots,v_n): v_1 \le v_2 \le \ldots \le v_n\}$.  {\em A word on our notation:} even though we had reduced to $X_{()}$, in the sequel, when we write $P_F$ or $E_F$, this means that the common distribution of the original $X_i$'s is $F$.
For measuring the content of a region for $\Delta$ we use Lebesgue measure $\nu$ on $\Re$. 

The first invariance reduction is obtained through location-invariance. The problem is invariant with respect to translations with the groups of transformations being, for every $c \in \Re$,
$$x_{()} \mapsto x_{()}+c;\ (F,\Delta) \mapsto (F,\Delta+c);\ \mbox{and}\
\theta \mapsto \theta+c.$$
A CR $\Gamma(X_{()})$ is location-equivariant if, for every $c \in \Re$,
%
$\Gamma(x_{()}+c) = \Gamma(x_{()}) + c,$
%
where $x_{()} + c = (x_{(1)}+c,\ldots,x_{(n)}+c)$. Observe that for a location-equivariant $\Gamma(\cdot)$, we have for every $c \in \Re$ that
\begin{eqnarray*}
\lefteqn{P_{(F,\Delta)}\{\Delta \in \Gamma(X_{()})\}  =  P_{(F,\Delta)}\{\Delta \in \Gamma(X_{()}+c) - c\} } \\
 & = &  P_{(F,\Delta)}\{\Delta + c \in \Gamma(X_{()}+c)\} 
 =  P_{(F,\Delta+c)}\{\Delta + c \in \Gamma(X_{()})\} \\
 & = & P_{(F,0)}\{0 \in \Gamma(X_{()})\}
\end{eqnarray*}
by taking $c = -\Delta$ to obtain the last equality. The problem has thus been reduced to considering $X_{()}$ to be the order statistics from $F(\cdot)$ and we seek a location-equivariant $\Gamma(x_{()})$ such that, for every $F \in \mathfrak{F}_{c,0}$,
%
$P_F\{0 \in \Gamma(X_{()})\} \ge 1 - \alpha.$
%
In addition, we seek to minimize $E_F \nu[\Gamma(X_{()})]$ over all $F \in \mathfrak{F}_{c,0}$. Note that Lebesgue measure $\nu$ in $\Re$ is location-invariant, that is, $\nu(B) = \nu(B+c)$ for every $B \in \mathfrak{B}$ and $c \in \Re$.

We remark at this stage that if we {\em know} the distribution $F(\cdot)$, then we could determine the optimal CR for $\Delta$ {\em under} this known distribution and no further invariance reduction will be needed. To demonstrate, suppose that $F$ is the normal distribution with mean zero and variance $\sigma^2$ which could be taken to be $\sigma^2 = 1$, so $F(\cdot) = \Phi(\cdot)$, with $\Phi(\cdot)$ the standard normal distribution function (we also let $\phi(\cdot)$ to denote the standard normal density function). Then, we seek a location-equivariant $\Gamma^*(x_{()})$ satisfying
%
$P_\Phi\{0 \in \Gamma^*(X_{()})\} \ge 1 - \alpha$ and with
$E_\Phi \int_{\Re} I\{w \in \Gamma^*(X_{()})\} dw$ minimized.
%
Under $\Phi$, the joint density function of $X_{()}$ is given by
$f(x_{()}) = n! \prod_{i=1}^n \phi(x_{(i)}) I\{x_{()} \in \mathfrak{X}\}$.
Thus, we want
$P_\Phi\{0 \in \Gamma^*(X_{()})\} = \int I\{0 \in \Gamma^*(x_{()})\} f(x_{()}) dx_{()} \ge 1 - \alpha.$
On the other hand, we obtain
\begin{eqnarray*}
\lefteqn{E_\Phi \int_{\Re} I\{w \in \Gamma^*(X_{()})\} dw  
=  \int_{\mathfrak{X}} \int_\Re I\{w \in \Gamma^*(x_{()})\} f(x_{()}) dw dx_{()} } \\
& = & \int_{\mathfrak{X}} \int_\Re I\{0 \in \Gamma^*(x_{()}-w)\} f(x_{()}) dw dx_{()} 
 =  \int_{\mathfrak{X}} I\{0 \in \Gamma^*(x_{()})\} h(x_{()}) dx_{()}
\end{eqnarray*}
where
$$h(x_{()}) = n! \frac{(2\pi)^{-(n-1)/2}}{\sqrt{n}} \exp\left\{-\frac{1}{2} \sum_{i=1}^n (x_{(i)} - \bar{x})^2\right\}$$
obtained after the obvious change-of-variables. The problem is then to find a location-equivariant $\Gamma^*(x_{()})$ that will minimize $\int_{\mathfrak{X}} I\{0 \in \Gamma^*(x_{()})\} h(x_{()}) dx_{()}$ subject to the condition that $\int I\{0 \in \Gamma^*(x_{()})\} f(x_{()}) dx_{()} \ge 1 - \alpha.$ The solution to this constrained minimization problem (see the optimization result in Theorem \ref{optimality with lks known}) is the well-known $z$-confidence interval for the normal mean given by
$\Gamma^*(x_{()}) = [\bar{x} \pm z_{\alpha/2} (1/\sqrt{n})].$

However, since $F$ is known only to belong to $\mathfrak{F}_{c,0}$, a further invariance reduction is needed. This is achieved through strictly increasing continuous transformations with $0$ as a fixed point. Let $\mathcal{M}$ denote the collection of functions $m(\cdot)$ that are strictly increasing continuous function on $\Re$ with $m(0) = 0$. The groups of transformations are given by
$$x_{()} \mapsto (m(x_{(1)}),m(x_{(2)}),\ldots,m(x_{(n)}));\ 
F \mapsto F m^{-1};\ \mbox{and}\
\theta \mapsto m(\theta).$$
$\Gamma(x_{()})$ is then equivariant with respect to these groups of transformations if
\begin{eqnarray*}
& \Gamma(m(x_{(1)}),\ldots,m(x_{(n)}))  =  m\Gamma(x_{(1)},\ldots,x_{(n)}) 
 \equiv   \{m(w): w \in \Gamma(x_{()})\}, &
\end{eqnarray*}
so that for every $m \in \mathcal{M}$ and $x_{()} \in \mathfrak{X}$, we have
$\Gamma(x_{()}) = m^{-1} \Gamma(m(x_{()})).$
We then have that, for every $m \in \mathcal{M}$ and $F \in \mathfrak{F}_{c,0}$,
\begin{eqnarray*}
P_F\{0 \in \Gamma(X_{()})\} & = & P_F\{0 \in \Gamma(m^{-1}m(X_{()}))\} 
 =  P_F\{0 \in m^{-1} \Gamma(m(X_{()}))\} \\
& = & P_F\{0 \in \Gamma(m(X_{()}))\ \mbox{since $m(0) = 0$}  \\
 & = &  P_{Fm^{-1}} \{0 \in \Gamma(X_{()})\}.
\end{eqnarray*}
Observe, however, that
\begin{eqnarray*}
& E_F \nu[\Gamma(X_{()})]  =  E_F \nu[m^{-1} \Gamma(m(X_{()}))] 
 =  E_{Fm^{-1}} \nu[m^{-1} \Gamma(X_{()})] &
\end{eqnarray*}
and we do {\em not} have in this situation quasi-invariance of the measure $\nu$ with respect to the groups of monotone transformations.

The group of transformations $\mathcal{M}$ with $F \mapsto Fm^{-1}$ is transitive over $\mathfrak{F}_{c,0}$. Thus, we may simply pick an arbitrary $F_0 \in \mathfrak{F}_{c,0}$, which could be taken to be $F_0 = U[-1,1]$, the uniform distribution over $[-1,1]$. Indeed, if $X \sim F \in \mathfrak{F}_{c,0}$, then with $m_F(v) = 2F(v) - 1$, we have $m(X) \sim U[-1,1]$. Thus,
$$P_F\{0 \in \Gamma(X_{()})\} =  P_{F_0}\{0 \in \Gamma(X_{()})\} \ \mbox{and} \
E_F \nu[\Gamma(X_{()})] =  E_{F_0} \nu[m^{-1} \Gamma(X_{()})].$$
We emphasize again that in the second equation we could {\em not} drop the term $m^{-1}$ nor factor it out from inside the $\nu(\cdot)$ measure. This will prevent us from obtaining a uniformly (over $\mathfrak{F}_{c,0}$) best confidence region for $\Delta$.

Next, we obtain a representation of $\Gamma(x_{()})$ by choosing a
specific member of $\mathcal{M}$ that depends on $x_{()}$. For an $x_{()}$, define $m(x_{()})(w)$ for $w \in \{x_{(i)} - x_{(n)}, i=1,2,\ldots,n\}$  via
$$m(x_{()})(w) = \sum_{i=1}^n I\{x_{(i)} - x_{(n)} \le w\} - n,$$
and for $w \in \Re$ define it such that it is strictly increasing and continuous over all $w \in \Re$. Observe that for $j=1,2,\ldots,n$,
$$m(x_{()})(x_{(j)}-x_{(n)})= j - n \quad \mbox{and} \quad m(x_{()})(-x_{(n)}) = B(x_{()}) - n,$$
where $B(x_{()}) = \sum_{j=1}^n I\{x_{(j)} \le 0\} = \sum_{j=1}^n I\{x_j \le 0\}$. Note that $m(x_{()})(0) = n - n = 0$ and
observe that $m(x_{()})^{-1}(j - n) = x_{(j)} - x_{(n)}, j=1,2,\ldots,n.$

\begin{lemm}
\label{lemm-representation}
With $m(x_{()})(\cdot)$ defined as above, a location-equivariant (LE) and $\mathcal{M}$-equivariant (ME) $\Gamma(x_{()})$ has representation
\begin{equation}
\label{representation of LE and ME}
\Gamma(x_{()}) = 
m(x_{()})^{-1} [\Gamma_0 - n] + x_{(n)}
\end{equation}
where $\Gamma_0$ is some region in $\Re$. Thus, the LE and ME $\Gamma(\cdot)$'s are determined by $\Gamma_0$'s, which are subsets of $\Re$. In fact, given a $\Gamma_0 \subset \Re$, we have
\begin{equation}
\label{Gammax in terms of Gamma0}
\Gamma(x_{()}) = \bigcup_{k \in [\Gamma_0 \cap \{0,1,\ldots,n\}]} \left[x_{(k)},x_{(k+1)}\right)
\end{equation}
whose Lebesgue measure is
%
$\nu[\Gamma(x_{()})] = \sum_{k \in [\Gamma_0 \cap \{0,1,\ldots,n\}]} \left[x_{(k+1)} - x_{(k)}\right].$
%
\end{lemm}

\begin{proof}
We utilize the location-equivariance (LE) and $\mathcal{M}$-equivariance (ME) of $\Gamma(\cdot)$. We have
\begin{eqnarray*}
\Gamma(x_{()}) & = & \Gamma(x_{()} - x_{(n)}) + x_{(n)} \ \mbox{(by LE)} \\
& = & m(x_{()})^{-1} \Gamma(m(x_{()})(x_{()} - x_{(n)})) + x_{(n)} \ \mbox{(by ME)} \\
& = & m(x_{()})^{-1}  \Gamma(1 - n, 2 - n, \ldots, (n-1) - n, n - n) + x_{(n)} \\
& = & m(x_{()})^{-1} [ \Gamma(1, 2, \ldots, n-1, n) - n] + x_{(n)}  \ \mbox{(again, by LE)} \\
& = & m(x_{()})^{-1} [ \Gamma_0 - n] + x_{(n)},
\end{eqnarray*}
where $\Gamma_0 = \Gamma(1,2,\ldots,n)$.
To establish (\ref{Gammax in terms of Gamma0}), given a $\Gamma_0 \subset \Re$, observe that
\begin{eqnarray*}
\{0 \in \Gamma(x_{()})\} & \Longleftrightarrow & \{ 0 \in m(x_{()})^{-1} [\Gamma_0-n] + x_{(n)} \} \\
 & \Longleftrightarrow  & \{ m(x_{()})(-x_{(n)}) \in \Gamma_0 - n\} \\
& \Longleftrightarrow & \{ (B(x_{()}) - n) \in (\Gamma_0 - n)\}  \\
& \Longleftrightarrow &  \{ B(x_{()}) \in \Gamma_0\}.
\end{eqnarray*}
It now follows that
\begin{eqnarray*}
\{w \in \Gamma(x_{()})\} & \Longleftrightarrow &  \{0 \in \Gamma(x_{()} - w)\} \ \mbox{[by LE property]} \\
& \Longleftrightarrow & \{B(x_{()} - w) \in \Gamma_0\} \ \mbox{[preceding result]} \\
& \Longleftrightarrow & \left\{ w \in \bigcup_{k \in \Gamma_0 \cap \{0,1,\ldots,n\}} \{v: B(x_{()}-v) = k\}\right\} \\
& \Longleftrightarrow & \left\{ w \in \bigcup_{k \in \Gamma_0 \cap \{0,1,\ldots,n\}} [x_{(k)},x_{(k+1)})\right\}.
\end{eqnarray*}
Thus, given a $\Gamma_0 \subset \Re$, $\Gamma(x_{()}) = \bigcup_{k \in \Gamma_0 \cap \{0,1,\ldots,n\}} [x_{(k)},x_{(k+1)})$ establishing (\ref{Gammax in terms of Gamma0}). The last result about the Lebesgue measure of $\Gamma(x_{()})$ is immediate since the intervals $\{[x_{(k)},x_{(k+1)}), k \in \Gamma_0 \cap \{0,1,\ldots,n\}\}$ are disjoint.
\end{proof}

\subsection{Optimal CRs}

Next, we tackle the problem of choosing an `optimal' (properly defined) region $\Gamma_0$, which then determines $\Gamma(x_{()})$ via the representation in Lemma \ref{lemm-representation}. Recall that the goal is to find $\Gamma(\cdot)$ such that with $F_0 = U[-1,1]$,
$P_{F_0}\{0 \in \Gamma(X_{()})\} \ge 1 - \alpha$
and, for every $F \in \mathfrak{F}_{c,0}$, $E_{F}\{\nu[\Gamma(X_{()})]\}$ minimized, or if this is not possible, made small. Note that under $F_0$, $B = B(X_{()}) = B(X) = \sum_{i=1}^n I\{X_i \le 0\}$ has a binomial distribution with parameters $(n,1/2)$, denoted by $\mathfrak{B}(\cdot;n,1/2)$, and with associated probability mass function 
$b(k;n,1/2) = {n \choose k} 2^{-n} I\{k \in \{0,1,\ldots,n\}\}.$ 
We have that
\begin{eqnarray*}
1 - \alpha \le  P_{F_0}\{0 \in \Gamma(X_{()})\} 
 =  P_{F_0}\{B \in \Gamma_0\} 
  =   \sum_{k=0}^n I\{k \in \Gamma_0\} b(k;n,1/2) 
 =  \sum_{k=0}^n \delta_0(k) b(k;n,1/2)
\end{eqnarray*}
with the first equality obtained using the portion of the proof of Lemma \ref{lemm-representation} and 
where we define 
$\delta_0(k) = I\{k \in \Gamma_0\}, k=0,1,\ldots,n.$
The expected Lebesgue measure of $\Gamma(X_{()})$ is
\begin{eqnarray*}
E_F \nu[\Gamma(X_{()})] & = & \sum_{k=0}^n I\{k \in \Gamma_0\} [E_F(X_{(k+1)}) - E_F(X_{(k)})] =
\sum_{k=0}^n \delta_0(k) l(k;F)
\end{eqnarray*}
where, with $X_{(0)} \equiv E_F[X_{(0)}] \equiv \inf\{v \in \Re: F(v) > 0\}$ and $X_{(n+1)} \equiv E_F[X_{(n+1)}] \equiv \sup\{v \in \Re: F(v) < 1\}$,
\begin{equation}
\label{defn lks}
l(k;F) = E_F(X_{(k+1)}) - E_F(X_{(k)}), k=0,1,\ldots,n.
\end{equation}
Assume {\em first} that we know the values of $\{l(k) \equiv l(k;F), k=0,1,\ldots,n\}$. We now allow for randomized confidence regions in order to achieve optimality, that is, we allow for $\Gamma_0$ and $\Gamma$ to depend on a randomizer $U$ which is a standard uniform random variable independent of the $X_i$'s. We remark that in \cite{Hoo82a,Hoo82b,Hoo86} randomized procedures were also allowed to enable achieving optimality, similarly to the Neyman-Pearson theory of most powerful tests (cf., \cite{LehRom05}). 

Define the right-continuous non-decreasing $[0,1]$-valued function, for $t \in \Re$,
\begin{eqnarray*}
G(t) & = & P\{b(B;n,1/2) > t l(B;F)\} 
 =  \sum_{k=0}^n I\{b(k;n,1/2) > t l(k;F)\} b(k;n,1/2).
\end{eqnarray*}
For a given $\alpha \in (0,1)$, define
$$c = \inf\{t: G(t) \le 1 - \alpha\} \quad \mbox{and} \quad
\gamma = \frac{(1-\alpha) - G(c)}{G(c-) - G(c)}.$$
Define the function $\delta_0^*$ over $\{0,1,\ldots,n\} \times [0,1]$ via
\begin{eqnarray*}
\delta_0^*((k,u))  & = & I\{b(k;n,1/2) > c l(k;F)\} + 
I\{b(k;n,1/2) = c l(k;F)\} I\{u \le \gamma\}.
\end{eqnarray*}
The optimal $\Gamma_0$ then satisfies
$\delta_0^*(k,u) = I\{k \in \Gamma_0^*(u)\}.$

\begin{theo}
\label{optimality with lks known}
Let $F \in \mathfrak{F}_{c,0}$ and let $\{l(k;F), k=0,1,\ldots,n\}$ be as defined in (\ref{defn lks}). Then
$E_{F} [\delta_0^*(B,U)] = 1 - \alpha.$
Furthermore, if $\delta_0$ is any other $\{0,1\}$-valued function in $\{0,1,\ldots,n\} \times [0,1]$ with
$E_{F} [\delta_0(B,U)] \ge 1 - \alpha,$
then
$$E \left\{\sum_{k=0}^n [\delta_0^*(k,U) l(k;F)]\right\} \le E\left\{\sum_{k=0}^n [\delta_0(k,U) l(k;F)]\right\},$$
where the expectation is with respect to the randomizer $U$.
\end{theo}

\begin{proof}
From the form of $\delta_0^*$, we have
\begin{eqnarray*}
E_F[\delta_0^*(B,U)]  & = & P\{b(B;n,1/2) > c l(B;F)\} + \gamma P\{b(B;n,1/2) = c l(B;F)\} \\
& = & G(c) + \left[\frac{(1-\alpha)-G(c)}{G(c-)-G(c)}\right] [G(c-) - G(c)] 
 =  1 - \alpha.
\end{eqnarray*}
Let $\delta_0$ be any other function on $\{0,1,\ldots,n\} \times [0,1]$ with
$E_F[\delta_0(B,U)]  \ge 1 - \alpha.$
From the definition of $\delta_0^*$, we observe that for each $(k,u) \in \{0,1,\ldots,n\} \times [0,1]$,
$$[b(k;n,1/2) - c l(k;F)][\delta_0^*(k,u) - \delta_0(k,u)] \ge 0.$$
Summing over $k = 0,1,\ldots,n,$ and integrating over $u \in [0,1]$, we find that
\begin{eqnarray*}
E_F\{\delta_0^*(B,U) - \delta_0(B,U)\} 
\ge c\left[\sum_{k=0}^n \int_0^1 l(k;F) \delta_0^*(k,u) du - \sum_{k=0}^n \int_0^1 l(k;F) \delta_0(k,u) du\right].
\end{eqnarray*}
Since $c \ge 0$ and by condition we have $E_F\{\delta_0^*(B,U) - \delta_0(B,U)\} \le 0$, then 
$$\sum_{k=0}^n \int_0^1 l(k;F) \delta_0^*(k,u) du \le \sum_{k=0}^n \int_0^1 l(k;F) \delta_0(k,u) du$$
which completes the proof of the theorem.
\end{proof}

\begin{rema}
 We note that this proof is similar to that of the Neyman-Pearson Lemma except for the fact that the $\{l(k;F): k=0,1,\ldots,n\}$ is not a distribution function.
 \end{rema}

Therefore, the optimal $\Gamma_0$, possibly using a randomizer $U \sim U[0,1]$, is 
\begin{eqnarray*}
\Gamma_0^*(u) & = & \left[\{k \in \{0,1,\ldots,n\}: b(k;n,1/2) > c l(k;F)\}\right]  \bigcup \\
&& \left[\{u \le \gamma\} \cap \{k \in \{0,1,\ldots,n\}: b(k;n,1/2) = c l(k;F)\}\right].
\end{eqnarray*}
The associated optimal confidence region for $\Delta$, possibly randomized, is
\begin{eqnarray}
\lefteqn{\Gamma^*(X_{()},U)  =  \left[\bigcup_{\{k \in \{0,1,\ldots,n\}:\  b(k;n,1/2) > c l(k;F)\}} [X_{(k)},X_{(k+1)})\right] \bigcup } \nonumber \\
&& \left[\{U \le \gamma\} \cap \left\{\bigcup_{\{k \in \{0,1,\ldots,n\}:\  b(k;n,1/2) = c l(k;F)\}} [X_{(k)},X_{(k+1)})\right\}\right].
\label{optimal CR known lks}
\end{eqnarray}
Note that if we drop the term in bracket containing $\{U\le \gamma\}$ in (\ref{optimal CR known lks}), we obtain
\begin{equation}
\label{conservative CR for median problem}
\Gamma^*(X_{()}) = \bigcup_{\{k \in \{0,1,\ldots,n\}:\  b(k;n,1/2) \ge c l(k;F)\}} [X_{(k)},X_{(k+1)}),
\end{equation}
which is a conservative confidence region for $\Delta$ in the sense that $P_{(F,\Delta)}\{\Delta \in \Gamma^*(X_{()})\} \ge 1 - \alpha.$ If $100(1-\alpha)\%$ is a {\em natural} confidence coefficient (cf., \cite{RanWol79}) associated with the binomial distribution, then we may still obtain an exact confidence region. 

\subsection{Implementation Aspects}


The optimal confidence region required knowledge of the $l(k;F)$'s, or at the very least the ordering of the $k \in \{0,1,\ldots,n\}$ for inclusion in the $\Gamma_0^*$, which is determined by the magnitude of the ratios $r(k;F) \equiv b(k;n,1/2)/l(k;F), k=0,1,\ldots,n$. In general, $l(k;F)$ will depend on the unknown true distribution $F$, hence the values of $l(k;F)$'s or the $r(k;F)$'s will not be known. In this case, we will not be able to determine $\Gamma^*(\cdot)$. We describe two approaches to circumvent this problem.
\begin{itemize}
\item[(i)] Restrict $F$ to belong to a subclass of the family of continuous distributions, say $\bar{\mathfrak{F}}_{c,0} \subset \mathfrak{F}_{c,0}$ and determine the $l(k;F)$'s or the $r(k;F)$'s for this class. This is then tantamount to satisfying the confidence level condition over the {\em whole} of $\mathfrak{F}_{c,0}$, but {\em focusing} only on the subclass $\bar{\mathfrak{F}}_{c,0}$ for minimizing the expected Lebesgue measure of the confidence region.
\item[(ii)] Utilize the observed data to estimate $l(k;F)$ by $\hat{l}(k)$, then use these $\hat{l}(k)$'s in the expression of $\Gamma^*(X_{()},U)$. There are several ways to accomplish this which are discussed below. However, it should be pointed out that when estimates are plugged-in, data double-dipping ensues and the achieved confidence level may not anymore satisfy the condition of being at least equal to $1-\alpha$.
\end{itemize}
The next two sections will deal with these two approaches towards developing the region estimators.

\section{Optimal CRs Focused on Symmetric Distributions}
\label{sect-symmetric distributions}

\subsection{For Uniform Distributions}

We illustrate the different approaches for dealing with the situation of unknown $l(k;F), k=0,1,\ldots,n$. Let us consider approach (i) first. Suppose that we consider the subfamily of uniform distributions. It suffices to consider $F = U[-\alpha,\alpha]$ for $\alpha > 0$. Note that $X \sim U[-\alpha,\alpha]$ iff $X \stackrel{d}{=} (2V - 1)\alpha$ with $V \sim U[0,1]$. Thus, $X_{(k)} \stackrel{d}{=} (2V_{(k)} - 1)\alpha$ and it is well-known that $E(V_{(k)}) = k/(n+1)$. As such $E(X_{(k)}) = (2k/(n+1)-1)\alpha$, hence
$$l(k;F) = \frac{2\alpha}{n+1} [ (k+1) - k] = \frac{2\alpha}{n+1}, k=0,1,\ldots,n.$$
Consequently, the ratios of interest are
$$r(k;F) = \frac{b(k;n,1/2)}{l(k;F)} = \frac{{n \choose k} 2^{-n}}{(2\alpha/(n+1))} \propto {n \choose k}, k=0,1,\ldots,n.$$
Hence, the optimal $\Gamma_0^*$ is of form
$\Gamma_0^* = \left\{k \in \{0,1,\ldots,n\}: {n \choose k} > c\right\}$
where $c$ is the smallest value such that $P\{B \in \Gamma_0^*\} \le 1-\alpha$. Since the mapping from $k \in \{0,1,\ldots,n\}$  into ${n \choose k}$ is symmetric about $n/2$ and decreases as $|k - n/2|$ increases, then $\Gamma_0^* = \{k \in \{0,1,\ldots,n\}: |k - n/2| < d\}$ for $d$ satisfying
\begin{eqnarray*}
d & = & \sup\left\{e: P\{|B - \frac{n}{2}| < e\} \le 1 - \alpha\right\} 
 =    \sup\left\{e: P\{B < \frac{n}{2} + e\} \le 1 - \alpha/2\right\}.
\end{eqnarray*}
This implies that $(n/2)+d$ is the $(1-\alpha/2)$th quantile of $\mathfrak{B}(\cdot;n,1/2)$, denoted by $b_{n,1/2;\alpha/2}$, obtainable via the {\tt qbinom} function in {\tt R} (\cite{R}). Letting $k_2 = b_{n,1/2;\alpha/2}$ and $k_1 = n - k_2$, then 
$$P\{k_1 < B < k_2\} \le 1 - \alpha \le P\{k_1 \le B \le k_2\}.$$
With 
$$\gamma = \frac{(1-\alpha) - P\{k_1 < B < k_2\}}{2\Pr\{B = k_2\}},$$
the resulting randomized confidence region for $\Delta$ is
\begin{equation}
\label{optimal for symmetric}
\Gamma_{10}^*(X_{()},U) = 
\left\{
\begin{array}{ccc}
\left.\left[X_{(k_1+1)}, X_{(k_2)}\right.\right) & \mbox{if} &  U > \gamma \\
\left.\left[X_{(k_1)}, X_{(k_2+1)}\right.\right) & \mbox{if} & U \le \gamma
\end{array}
\right..
\end{equation}
This $\Gamma_{10}^*$ CR procedure is the {\em randomized} version of the sign-statistic based CR, given in (\ref{sign stat CR}) and denoted by $\Gamma_3$. This $\Gamma_3$ CR was the confidence interval developed in \cite{Tho36}.

\subsection{For General Symmetric Distributions}

We assumed the uniform family of distributions in the preceding subsection. A question arises whether we obtain the same CR if $F$ belongs to the subfamily $\mathfrak{F}_{c,0}^{sym}$ of $\mathfrak{F}_{c,0}$, the additional condition being symmetry of the distribution. For instance, the classes of normal, Cauchy, logistic, double exponential, symmetric mixtures of distributions all belong to this subclass. It turns out that $\Gamma_{10}^*$ is also optimal for this larger subclass as a consequence of Theorem \ref{theo-about ratio}. The CR based on the Wilcoxon signed-rank statistic is an appropriate CR for these symmetric class of continuous distributions; however, this CR is {\em not} a legitimate competitor of $\Gamma_{10}^*(X_{()},U)$ since it does {\em not} satisfy the confidence level requirement for non-symmetric distributions which are still under the NMEM.

\begin{theo}
\label{theo-about ratio}
Let $X_1, \ldots, X_n$ be IID from $F \in \mathfrak{F}_{c,0}^{sym}$ and define $l(k;F) = E_F(X_{(k+1)}) - E_F(X_{(k)}), k=0, 1,\ldots,n$, with $X_{(0)} = \inf\{x \in \Re: F(x) > 0\}$ and $X_{(n+1)} = \sup\{x \in \Re: F(x) < 1\}$.
Let 
$$r(k;F) = \frac{{n \choose k}2^{-n}}{l(k;F)}, k=0,1,\ldots,n.$$
\begin{itemize}
\item[(i)] If $n$ is even, then $r(k;F)$ is maximized at $k = n/2$ and it decreases from the maximum as $|k - n/2|$ increases.
\item[(ii)] If $n$ is odd, then $r(k;F)$ is maximized at $k \in \{(n-1)/2,(n+1)/2\}$ and decreases when $(n-1)/2 - k$ increases for $k \le (n-1)/2$ or when $k - (n+1)/2 $ increases for $k \ge (n+1)/2$.
\end{itemize}
\end{theo}

The results stated in Theorem \ref{theo-about ratio} appear intuitive when the distribution $F \in \mathfrak{F}_{c,0}^{sym}$ is unimodal. What is surprising is that the results also hold when $F$ is bi-modal or multi-modal. For instance, if we take a symmetric mixture of a $N(-\mu,\sigma)$ and a $N(\mu,\sigma)$, even when $\mu$ is quite large relative to $\sigma$, the results still hold true. 
We present a mathematical proof of Theorem \ref{theo-about ratio}. To do so we first establish an identity analogous to that given by Pearson in a paper of \cite{Gal1902}. 

\begin{lemm}
\label{lemm-expression for lk}
Let $X_{(1)} < \ldots < X_{(n)}$ be the associated order statistics for a random sample of size $n$ from a continuous distribution $F$. For $k = 0, 1, 2, \ldots, n-1,n$, and provided that the expectations are well-defined with possibly a value of $\infty$, then
$$l(k;F) \equiv E_F(X_{(k+1)})-E_F(X_{(k)}) = {n \choose k} \int_{-\infty}^\infty F(x)^k  (1-F(x))^{n-k} dx.$$
\end{lemm}

\begin{proof}
Recall that for a positive-valued continuous random variable $W$, we have $E(W) = \int_0^\infty [1-F_W(w)] dw$, hence for a general continuous $W$, $$E(W) =  \int_0^\infty [1-F_W(w)] dw - \int_{-\infty}^0 F_W(w) dw.$$ For $k \in \{1,2,\ldots,n\}$, we also recall that $P\{X_{(k)} \le y \} = \sum_{j=k}^n {n \choose j} F(y)^k [1-F(y)]^{n-j}$. Using these expressions, we immediately find that
\begin{eqnarray*}
l(k;F) & = & E[X_{(k+1)}]  - E[X_{(k)}] = E[X_{k+1}^+ - X_{(k)}^+] - 
E[X_{k+1}^- - X_{(k)}^-] \\
& = & \int_0^\infty [(1-P\{X_{(k+1)} \le y\}) - (1-P\{X_{(k)} \le y\})]dy - \\
& & \int_{-\infty}^0 [P\{X_{(k+1)} \le y\} - P\{X_{(k)} \le y\}] dy \\
& = & {n \choose k} \left[\int_0^\infty F(y)^k [1-F(y)]^{n-k} dy + \int_{-\infty}^0 F(y)^k [1-F(y)^{n-k} dy\right] \\
& = & \int_{-\infty}^\infty F(y)^k [1-F(y)]^{n-k} dy.
\end{eqnarray*}
\end{proof}

We now prove Theorem \ref{theo-about ratio}.

\begin{proof}
Using the representation in Lemma \ref{lemm-expression for lk}, we have for $k = 0,1,\ldots,n$ that
\begin{displaymath}
r(k;F) = 2^{-n} \left\{\int_{-\infty}^\infty (1-F(x))^{n-k}F(x)^k dx\right\}^{-1}.
\end{displaymath}
Let $Q(k;F) = \int_{-\infty}^\infty (1-F(x))^{n-k}F(x)^k dx$. We prove the results by showing that $Q(k;F)$ is symmetric about $n/2$ and that it first decreases then increases with the maximum occurring at values of $k$ that depends on whether $n$ is odd or even. Making the transformation $u = F(x)$ in the expression of $Q(k;F)$ and denoting by $f$ the density function of $F$, we have
\begin{eqnarray*}
Q(k;F)  & = & \int_0^1 \frac{(1-u)^{n-k} u^k}{f[F^{-1}(u)]} du 
 = \int_0^{1/2} \frac{(1-u)^{n-k} u^k}{f[F^{-1}(u)]} du + 
\int_{1/2}^1 \frac{(1-u)^{n-k} u^k}{f[F^{-1}(u)]} du.
\end{eqnarray*}
In the last integral, let $w = 1 - u$ and note that since $F \in \mathfrak{F}_{c,0}^{sym}$, $F^{-1}(u) = -F^{-1}(1-u)$ so that $f[F^{-1}(u)] = f[-F^{-1}(u)] = f[F^{-1}(1-u)]$. Consequently,
\begin{eqnarray*}
Q(k;F) & = & \int_0^{1/2} \frac{(1-u)^{n-k} u^k + (1-u)^k u^{n-k}}{f[F^{-1}(u)]} du.
\end{eqnarray*}
Letting $c(u) = 1/f[F^{-1}(u)]$ and $D(k;F) = Q(k;F) - Q(k+1;F)$, it follows that
\begin{eqnarray*}
D(k;F) & = & \int_0^{1/2} c(u) (1-2u)\left[(1-u)^{n-k-1}u^k - (1-u)^ku^{n-k-1}\right] du \\
& = & \int_0^{1/2} c(u)(1-2u)(1-u)^{n-k-1}u^k \left[1 - \left(\frac{u}{1-u}\right)^{n-(2k+1)}\right] du.
\end{eqnarray*}
In the last integral, all terms in the integrand outside the brackets are nonnegative. Let
$b(u,\alpha) = 1 - [{u}/{(1-u)}]^\alpha$ for $\alpha \in \Re, u \in [0,1/2].$
For this function, we have
\begin{displaymath}
\lim_{u \downarrow 0} b(u;\alpha) = I\{\alpha > 0\} -\infty I\{\alpha < 0\} \quad \mbox{and} \quad b(1/2;\alpha) = 0 \ \mbox{all $\alpha$}.
\end{displaymath}
In addition, we have
\begin{displaymath}
b^\prime(u;\alpha) \equiv \frac{d}{du} b(u;\alpha) = 
-\alpha\left(\frac{u}{1-u}\right)^{\alpha-1} \frac{1}{(1-u)^2} 
\left\{
\begin{array}{ccc}
< 0 & \mbox{if} & \alpha > 0 \\
= 0 & \mbox{if} & \alpha = 0 \\
> 0 & \mbox{if} & \alpha < 0
\end{array}
\right..
\end{displaymath}
Therefore, on $u \in (0,1/2]$, $u \mapsto b(u;\alpha)$ is decreasing when $\alpha > 0$; constant at $0$ when $\alpha = 0$; and increasing when $\alpha < 0$. Consequently,
\begin{displaymath}
D(k;\alpha) = Q(k;F) - Q(k+1;F) = 
\left\{
\begin{array}{ccc}
< 0 & \mbox{if} & k < (n-1)/2 \\
= 0 & \mbox{if} & k = (n-1)/2 \\
> 0 & \mbox{if} & k > (n-1)/2
\end{array}
\right..
\end{displaymath}
This completes the proof of the theorem.
\end{proof}

\section{Optimal CRs Focused on Exponential Distributions}
\label{sect-exponential distributions}

Next, we consider the situation where the focused class of distributions is the negative exponential familiy, a right-skewed class of distributions, in contrast to the symmetric distributions considered above. Let $X_1,\ldots,X_n$ be IID from an exponential distribution (Exp($\lambda$)) $F(x;\lambda) = [1-\exp(-\lambda x)] I\{x \ge 0\}$ so the common median is $\Delta = \lambda^{-1} \log(2)$. From the normalized spacings theory (see, for instance, \cite{BarPro75}) we have that
$$X_{(k)} = \sum_{j=1}^k D_j/(n-j+1) \stackrel{d}{=} \sum_{j=1}^k X_j/(n-j+1), k=1,2,\ldots,n,$$
where $D_j = (n-j+1)[X_{(j)}-X_{(j-1)}], j=1,2,\ldots,n,$ are the normalized spacings statistics, which are also IID from Exp($\lambda$). As such, with $X_{(0)} \equiv 0$, we obtain
$$l(k;F) = E[X_{(k+1)}] - E[X_{(k)}] = \frac{1}{\lambda (n-k)}, k=0,1,2,\ldots,n.$$
It follows that under this negative exponential distribution model, 
\begin{displaymath}
r(k;F) = \frac{{n \choose k} 2^{-n}}{[\lambda(n-k)]^{-1}} \propto {n-1 \choose k}, k=0,1,2,\ldots,n.
\end{displaymath}
The optimal $\Gamma_0$ will therefore be of form
$$\Gamma_0^* = \left\{k \in \{0,1,\ldots,n\}: {n-1 \choose k} > c\right\}$$
with $c$ chosen to be the smallest value such that
$P\{B \in \Gamma_0^*\} \le 1 - \alpha.$
Note that this subset $\Gamma_0^*$ will never include $n$, but it could include $0$. We shall denote by $\Gamma_{11}^*(X_{()},U)$ the resulting randomized CR for $\Delta$ under this exponentially-distributed focused case.

We illustrate the resulting CR for concrete values of $n$. We start with $n = 10$, an even sample size. Using {\tt R} \cite{R} we obtain for each $k \in \{0,1,2,\ldots,n\}$ the values of $r(k;F)$ (up to proportionality) and $P\{B = k\}$ given in Table \ref{table-exponential example}.
\begin{table}
\caption{Values of $(k,r(k),P(B=k))$ for $k=0,1,2,\ldots,n$ for the case $n=10$ when $F$ is assumed to be a negative exponential distribution.}
\label{table-exponential example}
\begin{center}
\begin{tabular}{ccc} \hline
    $k$ &   $r(k)$ &    $P(B = k)$ \\ \hline
   0  & 1 & 0.0009765625 \\
   1   & 9 & 0.0097656250 \\
   2  & 36 & 0.0439453125 \\
   3  & 84 & 0.1171875000 \\
   4 & 126 & 0.2050781250 \\
   5 & 126 & 0.2460937500 \\
   6 & 84 & 0.2050781250 \\
   7  & 36 & 0.1171875000 \\
   8  & 9 & 0.0439453125 \\
  9  & 1 & 0.0097656250 \\
 10 &  0 & 0.0009765625 \\ \hline
\end{tabular}
\end{center}
\end{table}
Observe that the way we start including values of $k$ into $\Gamma_0^*$ is according to the value of $r(k)$. Thus we start by first including $k$-values of 4 and 5; then 3 and 6; etc. Observe the asymmetry in the process of including the $k$-values into $\Gamma_0^*$ with a bias in favor of the lower $k$-values. The intuition behind this is that since the exponential distribution is highly right-skewed, then the expected lengths between successive order statistics increases as $k$ increases, which is formally indicated by the $l(k;F) = 1/[\lambda(n-k)]$ expression. Thus, to shorten the interval, there is preference for the lower order statistics.
For a 95\% confidence region, from the table we find that $\Gamma_0^* = \{2,3,4,5,6,7\}$ which yields $P\{B \in \Gamma_0^*\} = .9346$. The randomization probability then becomes
$$\gamma = \frac{.95 - .9346}{.0537} = .2867.$$
The 95\% (randomized) confidence region for the median will then be
\begin{displaymath}
\Gamma_{11}^*(X,U) = [X_{(2)},X_{(8)}) I\{U > .2867\} + [X_{(1)},X_{(9)}) I\{U \le .2867\}.
\end{displaymath}
When $n = 11$, an odd sample size, by following the same calculations as for $n = 10$ and with the first $k$-value to enter being $k = 5$, we find the 95\% (randomized) confidence region for the median to be
\begin{displaymath}
\Gamma_{11}^*(X,U) = [X_{(3)},X_{(8)}) I\{U > .6758\} + [X_{(2)},X_{(9)}) I\{U \le .6758\}.
\end{displaymath}

\section{Data-Adaptive Methods}
\label{sect-adaptive methods}

Next we demonstrate approach (ii). In this situation we do not know the $l(k;F)$'s so we instead estimate these quantities using the observed data. Recall that $l(k;F) = E_F[X_{(k+1)} - X_{(k)}]$ so that without any knowledge about the underlying $F$ we will not have a closed-form expression for these $l(k;F)$'s. However, given the sample data, we could estimate $l(k;F)$, {\em unbiasedly}, by 
\begin{displaymath}
\hat{l}(k) = X_{(k+1)} - X_{(k)}, k=0,1,2,\ldots,n,
\end{displaymath}
which is the method-of-moments estimator. This estimator though maybe unstable or inefficient.
Using this estimator, we may then order the $k \in \{0,1,\ldots,n\}$ in terms of priority of entry into $\Gamma_0^*$ according to the quantities
\begin{displaymath}
\hat{r}(k) \propto \frac{{n \choose k}}{X_{(k+1)} - X_{(k)}}, k=0,1,2,\ldots,n. 
\end{displaymath}
As such the form of the `optimal' $\Gamma_0$ will be
\begin{displaymath}
\Gamma_0^* = \left\{k \in \{0,1,2,\ldots,n\}:  \frac{{n \choose k}}{X_{(k+1)} - X_{(k)}} > c\right\}.
\end{displaymath}
We shall denote by $\Gamma_{12}^*(X_{()},U)$ the resulting randomized CR.
The constant $c$ will be chosen to be the smallest value such that $P\{B \in \Gamma_0^*|X\} \le 1 - \alpha$. This value of $c$ will depend on the sample data $X$ since the ordering of entry of the $k$-values into $\Gamma_0^*$ will depend on $X$, so that the randomization probability $\gamma$ will also depend on $X$. Because of this dependence of the $c$ and $\gamma$ values to the sample data $X$, it is possible that the achieved confidence coefficient of the resulting confidence region will {\em not} anymore be $100(1-\alpha)\%$. In the simulation studies in section \ref{sec-Simulation Comparisons of Some of the Methods} we will indeed see that there is degradation in terms of the achieved confidence level for this CR.

A possibly better adaptive approach is to utilize the representation of $l(k;F)$ in Lemma \ref{lemm-expression for lk} and to replace the unknown $F(\cdot)$ in the expression by an estimator based on the $X_i-\hat{\Delta}, i=1,2,\ldots,n$, where $\hat{\Delta}$ is the sample median of the $X_i$'s. We expect this will lead to a better procedure since the estimators of the $l(k;F)$'s will be more stable compared to the method-of-moments estimator discussed above. In our implementation of this idea, we use as our estimator of $F$ the empirical distribution function $\hat{F}$ of $X_i - \hat{\Delta}, i=1,2,\ldots,n$, given by
$\hat{F}(t) = \frac{1}{n} \sum_{i=1}^n I\{X_i - \hat{\Delta} \le t\}, t \in \Re.$
The resulting estimator of $l(k;F)$ is then
\begin{eqnarray*}
\hat{l}(k) & = & {n \choose k} \int_{-\infty}^\infty (1-\hat{F}(x))^{n-k}\hat{F}(x)^k dx = 
{n \choose k} \sum_{i=2}^n \left[1 - \frac{i-1}{n}\right]^{n-k} \left[\frac{i-1}{n}\right]^k (X_{(i)} - X_{(i-1)}).
\end{eqnarray*}
Observe that we could also just have used the EDF of the $X_i$'s instead of the $(X_i-\hat{\Delta})$'s since the $\hat{\Delta}$ cancel out. As such, to order the $k \in \{0,1,\ldots,n\}$ in terms of entry into $\Gamma_0^*$, we use
$$\hat{r}(k) \propto \frac{{n \choose k}}{\hat{l}(k)} = \left\{\sum_{i=2}^n \left[1 - \frac{i-1}{n}\right]^{n-k} \left[\frac{i-1}{n}\right]^k [X_{(i)} - X_{(i-1)}]\right\}^{-1}.$$
The resulting randomized CR for $\Delta$ will be denoted by $\Gamma_{13}^*(X_{()},U)$.

These adaptive procedures are totally nonparametric in the sense that no knowledge of the underlying distribution $F$ is required. If we do have knowledge of the families of distributions in which $F$ belongs, then we may be able to provide a better estimate of the $l(k;F)$'s as in the preceding subsection. It should also be noted, however, that these adaptive procedures may not anymore satisfy the confidence level requirement due to the data-dependent plug-in step. Section \ref{sec-Simulation Comparisons of Some of the Methods}, which presents results of simulation studies, provides some insights regarding the empirical properties of these procedures.

\section{Brief Review of Existing `Off-the-Shelf' Median CRs}
\label{sec-current methods}

There are several existing methods for constructing frequentist-based $100(1-\alpha)\%$ CRs for $\Delta$ under the NMEM. We describe some of these methods prior to our illustrations and simulation studies. For a sample realization $\mathbf{x} = (x_1,\ldots,x_n)$, we define the usual sample statistics:
$$\bar{x} = \frac{1}{n} \sum_{i=1}^n x_i; \quad s^2 = s^2(\mathbf{x}) = \frac{1}{n-1} \sum_{i=1}^n (x_i-\bar{x})^2; \quad
\hat{\Delta} = m = m(\mathbf{x}) = \mbox{med}(\mathbf{x}).$$
$\bar{X}$, $S^2$, $\hat{\Delta}$, and $M$ will then represent the random versions of these sample statistics. As earlier mentioned, $\phi(z)$ and $\Phi(z)$ will be 
%
%
the standard normal density and distribution functions, respectively, $\Phi^{-1}(\cdot)$ the standard normal quantile function. We will use the conventional notation $z_\beta = \Phi^{-1}(1-\beta)$ for the $(1-\beta)$th quantile of $\Phi(\cdot)$.  $\mathcal{T}(\cdot;k)$ and $\mathcal{T}^{-1}(\cdot;k)$ will denote the distribution and quantile functions, respectively, of a Student's $T$ random variable with degrees-of-freedom $k$, and its $(1-\beta)$th quantile will be denoted by $t_{k;\beta}$.

Arguably, the most commonly-used CR for the center of a distribution, which is $\Delta$ for symmetric distributions, is the $T$-based CR given by
\begin{equation}
\label{t-based CR}
\Gamma_1(\mathbf{X}) = \left[\bar{X} \pm t_{n-1;\alpha/2} \frac{S}{\sqrt{n}}\right].
\end{equation}
However, this CR is actually {\em not} valid under the NMEM since it does not satisfy the condition $P_{(F,\Delta)}\{\Delta \in \Gamma(X)\} \ge 1 - \alpha$ for {\em all} $(F(\cdot),\Delta) \in \mathfrak{F}_{c,0} \times \Re$. We still included this CR since it is typically the first choice to use by practitioners when constructing a confidence interval for $\mu$ or $\Delta$ and we would like to examine and compare its performance with other CRs under the NMEM.

The nonparametric analog of the $T$-based CR is the CR constructed from the Wilcoxon signed-rank statistic $W^+$ (\cite{RanWol79}). The Walsh averages associated with the sample $X_i$'s are $W_{ij} = (X_i + X_j)/2$ for $i \le j$. Let $W_{(1)} \le W_{(2)} \le \ldots \le W_{((n)(n+1)/2)}$ be the order statistics of these Walsh averages. This Wilcoxon signed-rank based nonparametric CR is given by
\begin{equation}
\label{Wilcox-signed CR}
\Gamma_2(\mathbf{X}) =  [W_{(k_1+1)}, W_{(k_2+1)})
\end{equation}
where, with $\mathcal{W}^+(\cdot)$ denoting the null distribution (that is, under $\Delta = 0$ and $F \in \mathfrak{F}_{c,0}^{sym}$) of $W^+$, 
\begin{displaymath}
k_1 = \sup\{w: \mathcal{W}^+(w) \le \alpha/2\} \quad \mbox{and} \quad k_2 = \inf\{w: \mathcal{W}^+(w) \ge 1 - \alpha/2\}.
\end{displaymath}
This CR is valid under $F \in \mathfrak{F}_{c,0}^{sym}$, but not for $F \in \mathfrak{F}_{c,0}$. In contrast, the nonparametric CR derived from the sign statistic is valid under $\mathfrak{F}_{c,0}$ (see \cite{Tho36}). As before, let $\mathfrak{B}(\cdot)$ be the binomial distribution with parameters $n$ and $1/2$. Let 
\begin{displaymath}
k_1 = \sup\{w: \mathfrak{B}(w) \le \alpha/2\} \quad \mbox{and} \quad k_2 = \inf\{w: \mathfrak{B}(w) \ge 1 - \alpha/2\}.
\end{displaymath}
Then, this sign statistic-based CR is
\begin{equation}
\label{sign stat CR}
\Gamma_3(\mathbf{X}) = [X_{(k_1+1)}, X_{(k_2+1)}).
\end{equation}

Another CR of $\Delta$ is developed from the asymptotic normality of the sample median $M(\mathbf{X})$. For $X_1, X_2, \ldots, X_n$ IID from a distribution $F(\cdot)$ with density $f(\cdot)$, this asymptotic distribution (cf., \cite{RanWol79}) is given by
$$M(\mathbf{X}) \sim AN\left[\Delta,\frac{1}{n} \frac{1}{4 f(\Delta)^2}\right].$$
If $\hat{f}(\hat{\Delta};\mathbf{X})$ is an estimator of $f(\Delta)$, then an asymptotic confidence interval for $\Delta$ is
\begin{equation}
\label{sample median based 1}
\Gamma_4(\mathbf{X}) = \left[M(\mathbf{X}) \pm z_{\alpha/2} \left(\frac{1}{\sqrt{n} 2 \hat{f}(\hat{\Delta};\mathbf{X})}\right)\right].
\end{equation}
Kernel-based estimators of the density $f$ are available in {\tt R} using the {\tt density} object function (\cite{R}), and coupled with the {\tt approx} function, $f(\Delta)$ could then be estimated. This is how we implemented this CR in the illustrations and in the simulation studies.

Instead of relying on asymptotic approximations, one may resort to bootstrapping approaches. Let $\mathbf{X}_k^* = (X_{k1}^*,\ldots,X_{kn}^*)$ be the $k$th bootstrap sample out of BREPS bootstrap samples. Denote its sample median by $M_k^*$ .  The basic bootstrap CR using the median obtains the  $\alpha/2$th and $(1-\alpha)/2$th quantiles of $\{{M}^*_k - {M}, k=1,2,\ldots,BREPS\}$, denoted by $\kappa_{1-\alpha/2}^*$ and $\kappa_{\alpha/2}^*$, respectively, and constructs the CR (cf., \cite{EfrHas16}) via
\begin{equation}
\label{basic bootstrap CR}
\Gamma_5(\mathbf{X}) = [M(\mathbf{X}) - \kappa_{\alpha/2}^*, M(\mathbf{X}) - \kappa_{1-\alpha/2}^*].
\end{equation}

The next bootstrapped CR is derived using the studentized median as pivot and with its standard error estimated by $S_{boot}^*$,  the standard deviation of  $\{M_k^*, k=1,2,\ldots, BREPS\}$. The CR is constructed via
\begin{equation}
\label{median based with bootstrap}
\Gamma_6(\mathbf{X}) = \left[ M(\mathbf{X}) \pm t_{n-1;\alpha/2} (S_{boot}^*) \right].
\end{equation}

The bootstrap percentile CR also uses the bootstrap distribution of $M(\mathbf{X})$. Denoting by 
$$\mathcal{B}^*(t) = \frac{1}{BREPS} \sum_{k=1}^{BREPS} I\{M_k^* \le t\}$$
the (empirical) bootstrap distribution, the percentile bootstrap CR is
\begin{equation}
\label{percentile bootstrap CR}
\Gamma_7(\mathbf{X}) = \left[{\mathcal{B}^*}^{-1}(\alpha/2), {\mathcal{B}^*}^{-1}(1-\alpha/2)\right]
\end{equation}
where ${\mathcal{B}^*}^{-1}(\cdot)$ is the bootstrap quantile function.

This percentile bootstrap CR is usually bettered by so-called bias-corrected CRs. See, for instance, chapter 5 of \cite{DavHin97} and chapter 11 of \cite{EfrHas16}. The first improvement is provided by the bias-corrected (BC) CR which is
\begin{equation}
\label{BC CR}
\Gamma_8(\mathbf{X}) = \left[{\mathcal{B}^*}^{-1}\left(\Phi(2z_0 + \Phi^{-1}(\alpha/2))\right), {\mathcal{B}^*}^{-1}\left(\Phi(2z_0 + \Phi^{-1}(1-\alpha/2))\right)\right]
\end{equation}
where $p_0 = \mathcal{B}^*(M(\mathbf{X}))$ and $z_0 = \Phi^{-1}(p_0)$. On the other hand, the BCa method (bias-corrected and accelerated) CR takes the form
\begin{eqnarray}
 \Gamma_9(\mathbf{X}) & =   &
\left[{\mathcal{B}^*}^{-1}\left(\Phi\left[z_0 + \frac{z_0+ \Phi^{-1}(\alpha/2)}{1-a(z_0+ \Phi^{-1}(\alpha/2))}\right]\right), \right. \nonumber
\\
&& 
 \left. {\mathcal{B}^*}^{-1}\left(\Phi\left[z_0 + \frac{z_0+ \Phi^{-1}(1-\alpha/2)}{1-a(z_0+ \Phi^{-1}(1-\alpha/2))}\right]\right)\right]
 \label{BCa CR}
\end{eqnarray}
where the acceleration coefficient $a$ can be estimated using jackknifed samples estimates $\hat{M}_{(i)}, i=1,2,\ldots,n,$ via
$$\hat{a} = \frac{1}{6} \frac{\sum_{i=1}^n (\hat{M}_{(i)} - \hat{M}_{()})^2}
{\left[\sum_{i=1}^n (\hat{M}_{(i)} - \hat{M}_{()})^2\right]^{3/2}}$$
with $\hat{M}_{()} = \frac{1}{n} \sum_{i=1}^n \hat{M}_{(i)}$ with $\hat{M}_{(i)}$ the sample median of the $i$th jackknifed sample $$X_{-i} \equiv (X_1,\ldots,X_{i-1},X_{i+1},\ldots,X_n).$$

Together with these existing CRs for $\Delta$, we add the CRs developed in earlier sections (we drop the superscript `$^*$' in the notation): $\Gamma_{10}$, the CR optimized for symmetric distributions [see equation (\ref{optimal for symmetric})]; $\Gamma_{11}$, the CR optimized for exponential distributions (see section \ref{sect-exponential distributions}); $\Gamma_{12}$, the adaptive CR using the crude estimator for $l(k;F)$; and $\Gamma_{13}$, the adaptive CR using the empirical distribution in the estimator of $l(k;F)$ (see section \ref{sect-adaptive methods}). 
We summarize these different CR procedures examined in the illustrations and the simulation studies in Table \ref{table-methods list}. We note that among these thirteen CRs, those that are location-equivariant and equivariant to monotone transformations are $\Gamma_3$, $\Gamma_{10}$, $\Gamma_{11}$, $\Gamma_{12}$, and $\Gamma_{13}$.

\begin{table}
\caption{List of the confidence regions (CRs) considered in the illustrations and simulations.}
\label{table-methods list}
\begin{center}
\begin{tabular}{cl}
{\bf Method Label} & {\bf Description or Type or Basis of CR} \\ \hline
$\Gamma_1$ & $T$ Statistic Based \\
$\Gamma_2$ & Wilcoxon Signed-Rank Statistic Based\\
$\Gamma_3$ & Sign Statistic Based\\
$\Gamma_4$ & Asymptotic Distribution of Median Based\\
$\Gamma_5$ & Basic Median Bootstrap \\
$\Gamma_6$ & Median with Bootstrapped SE \\
$\Gamma_7$ & Percentile Bootstrapped \\
$\Gamma_8$ & Bias-Corrected (BC) Bootstrapped\\
$\Gamma_9$ & Bias-Corrected Accelerated (BCa) Bootstrapped \\
$\Gamma_{10}$ & Optimally Focused for Symmetric \\
$\Gamma_{11}$ & Optimally Focused for Exponential\\
$\Gamma_{12}$ & Adaptive with $l(k)$ Unbiasedly Estimated\\
$\Gamma_{13}$ & Adaptive with $l(k)$ Estimated using EDF\\ \hline
\end{tabular}
\end{center}
\end{table}

\section{Illustration using Real Data Sets}
\label{sec-illustration}

For our first illustration we use a real data set gathered by the first author about the mileage efficiency of his car during a period when he was commuting between Ann Arbor, Michigan and Bowling Green, Ohio. Efficiency is measured in terms of the miles traveled per gallon of gasoline. In the data set there were $n = 205$ observations, with each observation recorded at each gas fill-up.  The histogram of these observations are shown in Figure \ref{fig-car mileage}. The full data set, which contained other relevant variables, was also used in demonstrating linear model global model validation procedures in \cite{PenSla06}.

\begin{figure}
\caption{Histogram of 205 observations for the variable Average Miles Per Gallon ({\tt AveMilesGal}) for the data set gathered by the first author pertaining to his car.}
\label{fig-car mileage}
\includegraphics[width=\textwidth,height=2.25in,clip=]{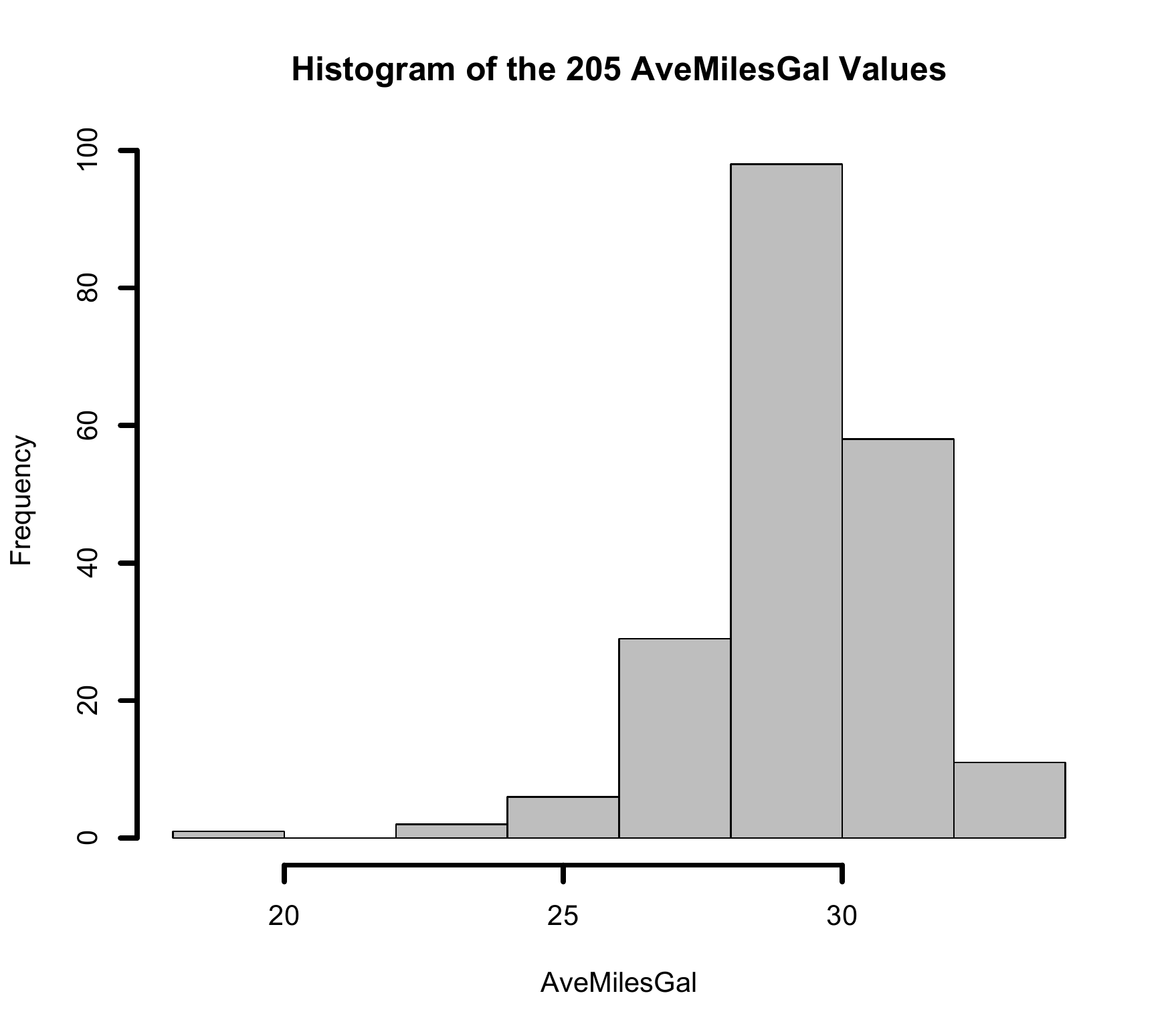}
\end{figure}

For this sample data the sample mean is $\bar{x} = 29.291$, sample median is $m = m(x) = \hat{\Delta} = 29.462$, the first and third sample quartiles are 28.275 and 30.575, respectively, and the extreme values are 18.374 and 33.472. From the histogram we notice that there is a noticeable left-skewness in the distribution. The sample standard deviation is 1.887 and the inter-quartile range is 2.299. For our illustration, we seek a confidence region for the population median $\Delta$, where the population could be thought of as the hypothetical set of all values of {\tt AveMilesGal} if the car has been observed forever. We mention that one interesting aspect of this data set is that the population from which this sample is obtained is a mixture of subpopulations corresponding to summer highway driving, winter highway driving, summer city driving, winter city driving, and a combination of highway and city driving. 

Figure \ref{fig-plot CR AveMilesGal} depicts the 95\% confidence regions produced by the thirteen different methods for the {\tt AveMilesGal} data set. 
\begin{figure}
\caption{The 95\% confidence regions for the median of the {\tt AveMilesGal} population produced by the thirteen different methods in Table \ref{table-methods list}. The horizontal line depicts the sample median.}
\label{fig-plot CR AveMilesGal}
\includegraphics[width=\textwidth,height=2.5in]{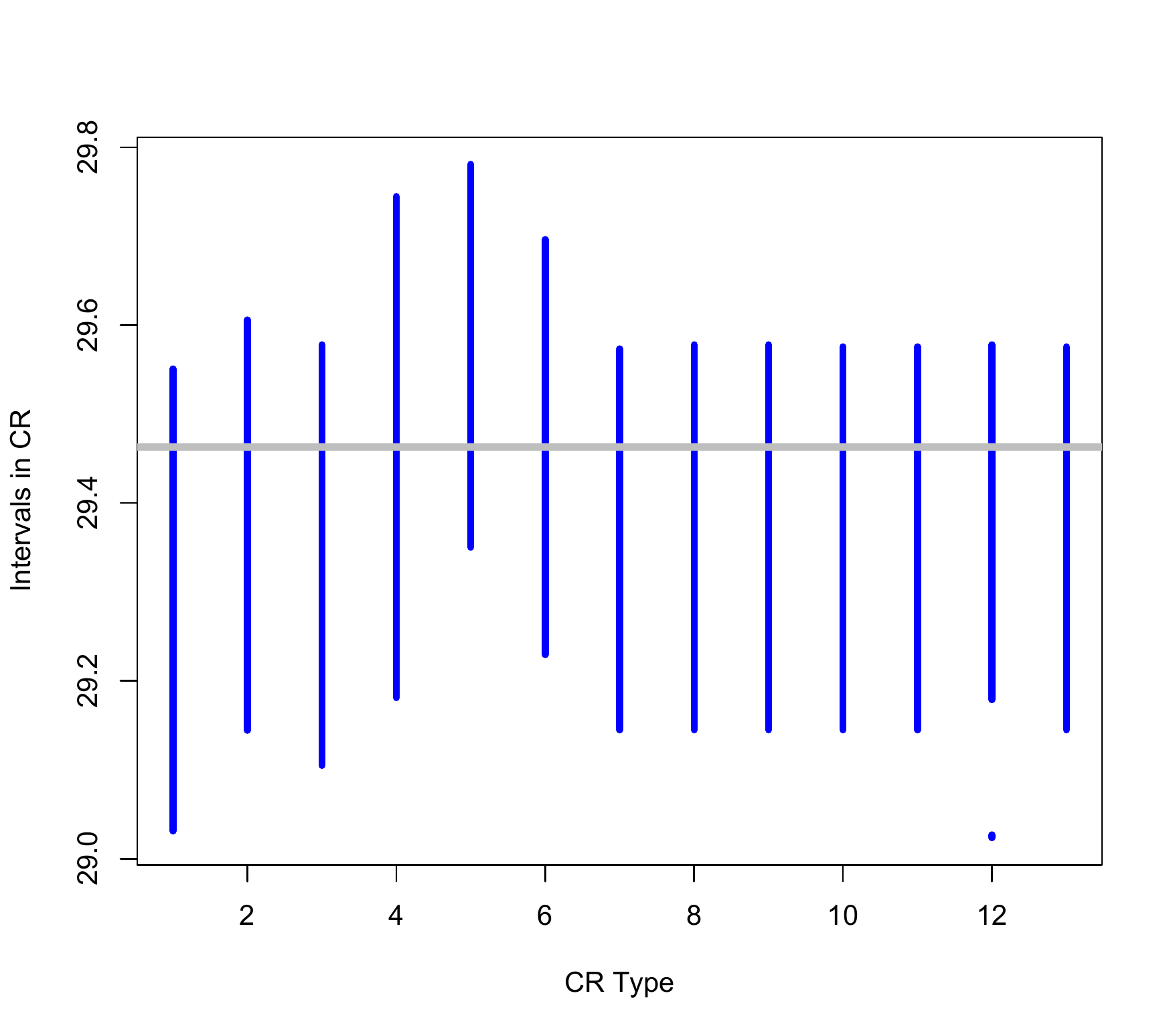}
\end{figure}
Except for CR $\Gamma_{12}$, all of them produced intervals. It is possible that method $\Gamma_{13}$ may also produce a region that is not an interval, while CRs $\Gamma_1$ to $\Gamma_{11}$ will all produce confidence {\em intervals} by construction. In addition, observe that only the CRs produced by $\Gamma_4$ and $\Gamma_6$ turned out to be symmetric about the sample median.

Our other illustration uses Proschan's famous Boeing air-conditioning data set \cite{Pro63}. Figure \ref{hist-ac data} 
\begin{figure}[h]
\caption{Histogram of Proschan's Boeing air-conditioning data set. The 212 observations are times between failures, and there were several airplanes that were monitored.}
\label{hist-ac data}
\includegraphics[width=\textwidth,height=2.25in,clip=]{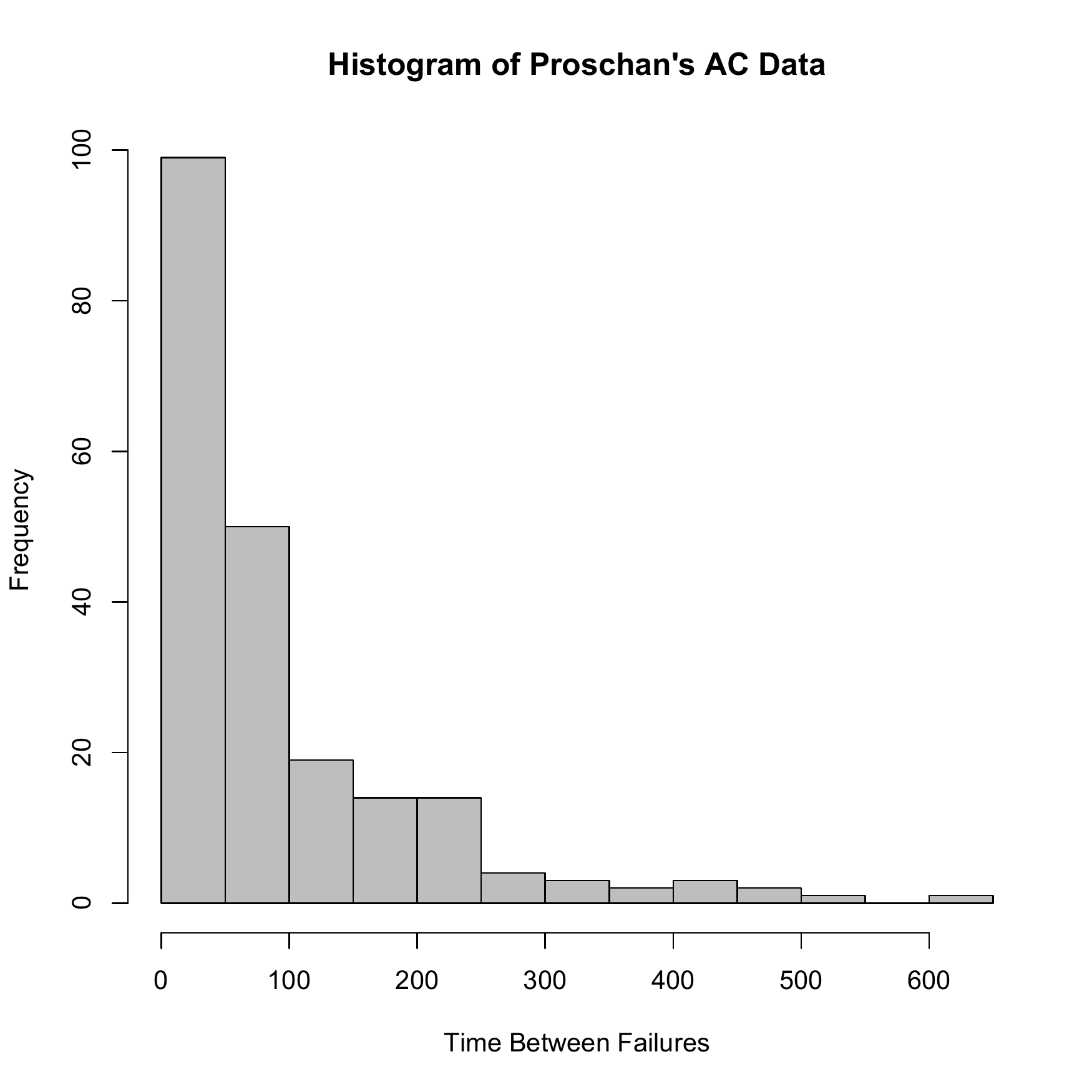}
\end{figure}
provides a histogram of the 212 observations in this data set, with each observation being a time between successive failures. To un-tie tied observations (we perturb the original values in tied sets of observations by uniform variates over $[-.001,.001]$. Note the high right-skewness of this data set. Proschan also demonstrated that this data set did not come from an exponential distribution, but came from a mixture of exponential distributions, inducing a decreasing failure rate property. Our goal in this illustration is to provide a confidence region for the median inter-failure time given this data set. Figure \ref{CRs for AC Data}
\begin{figure}[h]
\caption{The 95\% confidence regions for the median based on the thirteen methods for the Proschan's Boeing air-conditioning data set. The horizontal line represents the sample median.}
\label{CRs for AC Data}
\includegraphics[width=\textwidth,height=2.5in,clip=]{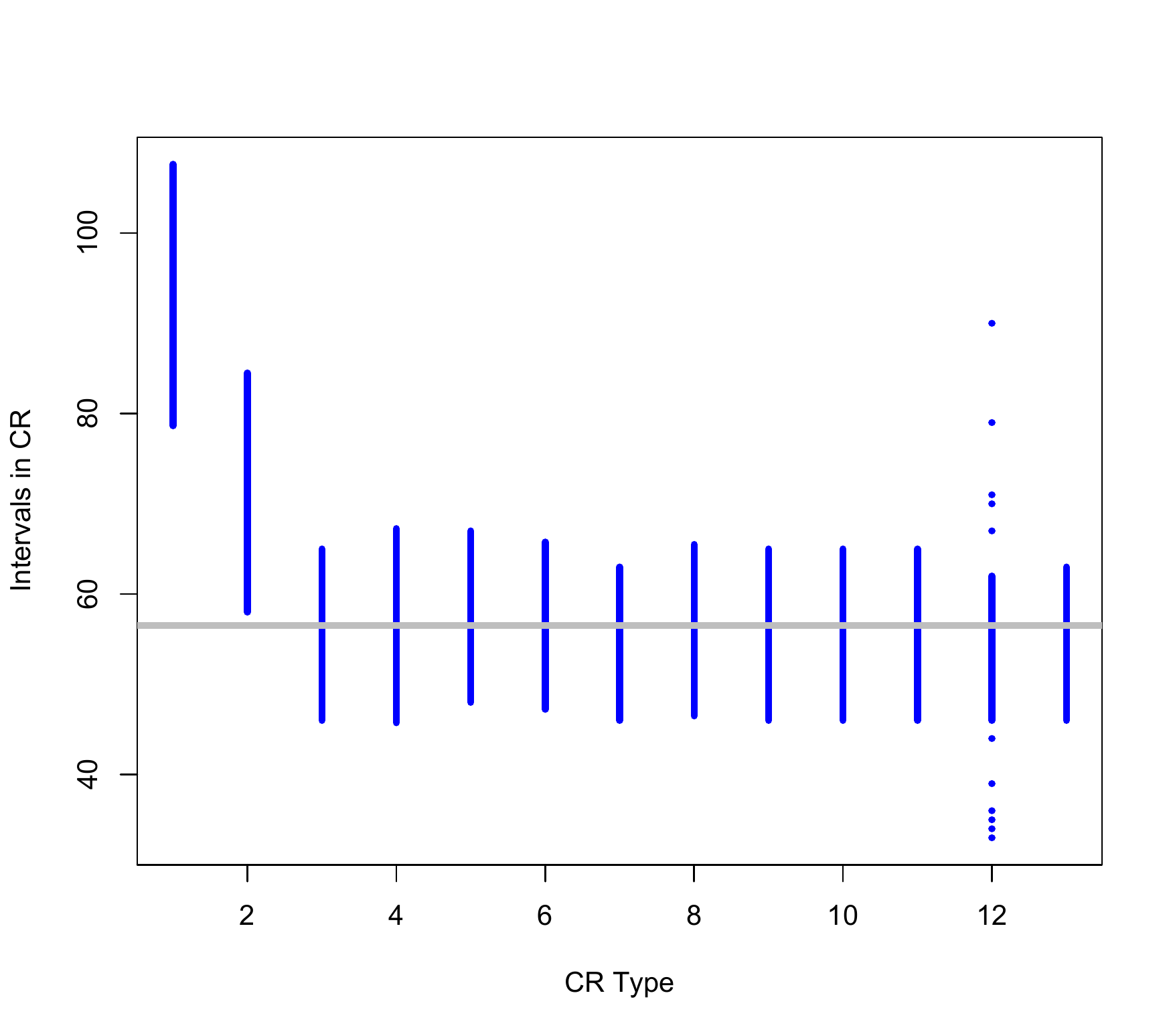}
\end{figure}
 presents the comparative plots of the CRs for the thirteen methods. Notice that the $T$-based CR, as well as the Wilcoxon signed-rank based CR, appear to be rather different from the other eleven CRs.  In fact, both excluded the sample median. At the same time, we are somewhat unfair in the manner in which we applied these two procedures since it would have been more appropriate to first perform a transformation to approximately symmetrize the distribution prior to applying these procedures, and then re-transforming back. Also, notice that the $\Gamma_{12}$ CR is composed of one big interval together with several smaller intervals. This could be attributed to the instability of the method-of-moments estimates of the $l(k;F)$'s. Interestingly, this CR has the smallest content, but we will see in the simulation studies in section \ref{sec-Simulation Comparisons of Some of the Methods} that its achieved coverage probability of the median tends to be lower than the nominal coverage probability.

\section{Comparison of Methods via Simulations}
\label{sec-Simulation Comparisons of Some of the Methods}

In this section we present the results of simulation studies to compare the thirteen CR methods listed in Table \ref{table-methods list} in terms of their contents and coverage probabilities under the NMEM.  Different error distributions [normal, Cauchy, uniform, logistic, gamma, Weibull and mixture of normals] and varied sample sizes ($n \in \{10,15,20,25,30,40,50\}$) were utilized in the simulations. The computer programs were coded in {\tt R} \cite{R}. For each combination of error distribution and sample size, 10000 simulation replications were ran, and with 2000 bootstrap replications for the bootstrap-based methods.

Figures \ref{fig-big-simuls} and \ref{fig-big-simuls-continued} present the results of this simulation study in the form of plots of the standardized mean contents (these are the simulated mean contents multiplied by the square root of the sample size) and the percentage coverage of the true median under different distributions and for the varied sample sizes. As we progress with the different distributions we eliminated in the plots the methods that were not competitive, either with respect to their achieved coverage probability or their mean content. This is in order to make the plots leaner and cleaner and to enable more clear-cut comparisons among those methods that are competitive. Methods that were eliminated from contention {\em after} a distribution run are as follows (they were still included in the runs but eliminated only in the plots): normal [5, 12 due to coverage]; Cauchy [1 due to content]; uniform [4 due to coverage]; logistic [none]; gamma [2 due to coverage]; Weibull [8 since almost the same as 9]; normal mixture [6 due to coverage]. Upon final examination of the resulting plots we also eliminated method 11 since it is unstable with respect to content. 

The first observation from these simulations concerns the performance of $\Gamma_1$ and $\Gamma_2$. Both perform very well under the normal distribution with $\Gamma_1$ having smaller content as expected, but $\Gamma_1$ performs extremely poorly with respect to content under the Cauchy distribution, and both do not perform well in terms of coverage for non-symmetric distributions. However, their poor performances under non-symmetric distributions are also expected. 
The next observation is the consistent validity of the sign-statistic based procedure $\Gamma_3$. Its coverage probability is always at least the specified coverage probability and in fact it tends to be conservative. At the same time, it has the highest mean content among the competitive methods. $\Gamma_{10}$, which is the procedure focused towards symmetric distributions and which is actually a randomized version of $\Gamma_3$, is preferable over $\Gamma_3$ since it has smaller mean content and at the same time it satisfies the coverage requirement. 
The basic bootstrap CR, $\Gamma_5$, surprisingly performed poorly with respect to its achieved coverage probability. The CR based on the studentized sample median using a bootstrapped standard deviation as estimate of the median's standard error, $\Gamma_6$, did not perform well for small sample sizes in terms of coverage probability. $\Gamma_8$ and $\Gamma_9$ almost have the same performance, as is to be expected from their derivations, while $\Gamma_{11}$ tended to have unstable mean contents. $\Gamma_{12}$ did not perform acceptably due to its degraded coverage probability. Thus, based on these simulation studies, the most competitive among the thirteen CR procedures are $\Gamma_3$, $\Gamma_{7}$, $\Gamma_{9}$, $\Gamma_{10}$, and $\Gamma_{13}$.

To focus comparisons among these five procedures, we ran more extensive simulations with 20000 replications and 5000 bootstrap replications for sample sizes $n \in \{10,15,20,25,30,40,50,75,100\}$ to obtain better comparisons of these five chosen methods. The results of these simulations are contained in Figures \ref{fig-big-simuls-chosen} and \ref{fig-big-simuls-chosen-continued}. 
Examining the plots in Figures \ref{fig-big-simuls-chosen} and \ref{fig-big-simuls-chosen-continued}, we see that the sign-based test, $\Gamma_3$, and the procedure focused on symmetric distributions, $\Gamma_{10}$, are consistently satisfying the confidence level requirement among all the distributions considered, including the skewed gamma and Weibull distributions, with $\Gamma_3$ tending to be more conservative. $\Gamma_{10}$, which as pointed out earlier is a randomized version of $\Gamma_{3}$, has confidence level that is very close to the pre-specified level, hence it also has smaller content compared to $\Gamma_3$. The PCT-Bootstrap, $\Gamma_7$, and the BCa-Bootstrap, $\Gamma_9$, tend to be liberal, and so is the adaptive $\Gamma_{13}$, though the latter could sometimes be conservative such when the distribution is Cauchy, or it could be more liberal than $\Gamma_7$ and $\Gamma_9$ such as for the mixture of normal distributions. Content-wise, $\Gamma_7$, $\Gamma_9$, and $\Gamma_{13}$ almost have the same performances, with $\Gamma_9$ appearing to be having a tad smaller content. Due to the lower achieved confidence, these three CRs also tended to have smaller contents compared to $\Gamma_3$ and $\Gamma_{10}$, with $\Gamma_3$ possessing the largest content and being also the most conservative. We note here the impact of the plug-in procedure for $\Gamma_{13}$. Since we did not know $l(k;F)$ we plugged-in an estimator for this function which utilized the empirical distribution function and which also used the same sample data, so there is data double-dipping. The impact of this plug-in and double-dipping procedure is to make the procedure somewhat liberal. In contrast, the $\Gamma_{10}$ CR did not need this plug-in procedure and we observe that the achieved level for this CR is very close to the pre-specified confidence level.

\section{Concluding Remarks}
\label{sec-conclusion}

In this paper we revisit the problem of constructing confidence regions (CRs) for the median under the NMEM. The problem is also relevant in the study of complex engineering systems where it is difficult to determine the exact functional form of the system's lifetime distribution, so one is forced to utilize a nonparametric model for the system's lifetime distribution. In addition, this problem also arises in economic settings where of interest will be the median instead of the mean when dealing with the income variable due to high right-skewness of its distribution. Several existing nonparametric approaches to constructing CRs for the median were reviewed.  Among these `off-the-shelf' methods are the $T$-based method which utilizes the sample mean, and that based on the Wilcoxon signed-rank statistic. These two methods may perhaps be the default methods for many users whose goal is to infer about the `center' of the population distribution. Also included are the computationally-intensive CR methods such as the bootstrap percentile and the BCa approaches (\cite{EfrHas16}). 

When assessing the adequacy of a CR, there is a need to examine its expected content (in one-dimensional settings, content is its Lebesgue measure) in addition to the requirement that it satisfies a specified nominal level of coverage of the true parameter value. Thus we  examined the problem of obtaining optimal CRs, which are those with smallest expected contents. Under the NMEM, invariance with respect to location-shifts and monotone transformations were invoked to reduce the problem to constructing equivariant CRs. Within the class of equivariant CRs, we obtained the best CRs, with respect to minimizing the expected content for subclasses of the class of all continuous distribution functions. Thus, there is a best CR when focused on the subclass of symmetric distributions and a best CR for the subclass of exponential distributions. We also developed fully nonparametric data-adaptive CRs under the NMEM. These adaptive CRs are the ones labeled $\Gamma_{12}$ and $\Gamma_{13}$. Based on simulation studies to compare the different CRs, we found that under the NMEM, both the $T$-based and the Wilcoxon signed-rank statistic based CRs should {\bf not} be utilized. The sign-statistic based procedure, $\Gamma_3$, and the optimal procedure focused towards symmetric distributions, $\Gamma_{10}$, both fulfill the confidence level requirement, with $\Gamma_3$ being somewhat conservative, hence tended to have larger content. Among the other CR methods, the bootstrap procedures $\Gamma_7$ (PCT) and $\Gamma_9$ (BCa), and the adaptive procedure $\Gamma_{13}$, were the most competitive among all scenarios, but these three tended to be a tad more liberal than either $\Gamma_3$ and $\Gamma_{10}$, hence also possessed smaller contents. If one is to insist that the desired confidence level should be achieved, then $\Gamma_{10}$ appears to be the best, or if one is opposed to using randomized CRs, then $\Gamma_3$ should be chosen. However, we feel that abhorrence to the use of randomized procedures is not mathematically justifiable nor defensible, since the use of randomized procedures allows for better methods and such strategies are in fact the bulwark of statistical decision and game theory. On the other hand, if one could tolerate a small degradation in the achieved confidence level, then $\Gamma_9$ and $\Gamma_{13}$ appear to be reasonable choices among these different CR procedures.

Finally, we mention that the approach of developing the CR procedures using invariance considerations is extendable to other more complex settings, such as two-sample settings, $K$-sample settings, situations with censored and/or truncated data, and when the parameter of interest is not necessarily the median. In addition, of particular and major interest to us leading to this work is the potential that the use of invariance arguments may enable the construction of optimal {\em multiple} confidence regions, extending results in multiple testing settings. We hope to explore these extensions and new possibilities in future work.

\section*{Acknowledgements}

The authors thank Professor James Lynch, Dr.\ Shiwen Shen, and graduate students Tahmidul Islam, Jeff Thompson, Lili Tong and Lu Wang for helpful discussions during research seminars. In addition, we thank Professors Laurent Doyen, Olivier Gaudoin  and {\'E}ric  Marchand for helpful comments during the first author's seminar presentation of this work at the {\'U}niversite Grenoble Alpes in Grenoble, France. We thank also
Professor Xianzheng Huang of the University of South Carolina for helpful comments. 

\section*{Funding}

Partially supported by NIH Grant P30GM103336-01A1. 

%


\baselineskip=12pt
\bibliography{ConfidenceRegions}



\newcommand{\myheight}{1.4in}

\begin{figure}[h]
\caption{Plots of the results of simulation studies with 10000 replications over a set of sample sizes ($n \in \{10,15, 20, 25, 30, 40, 50\}$) for normal, Cauchy, uniform, logistic, gamma(2,1), Weibull(.5,1), and mixture of normals error distributions. The plots pertaining to the lengths (contents) utilized the {\em standardized mean length}, which is the mean length multiplied by $\sqrt{n}$. The description of the 13 CR methods are in Table \ref{table-methods list}.}
\label{fig-big-simuls}
\begin{tabular}{c} \hline \\
{\bf Distribution:} Normal(0,1) \\
\includegraphics[width=\textwidth,height=\myheight,clip=]{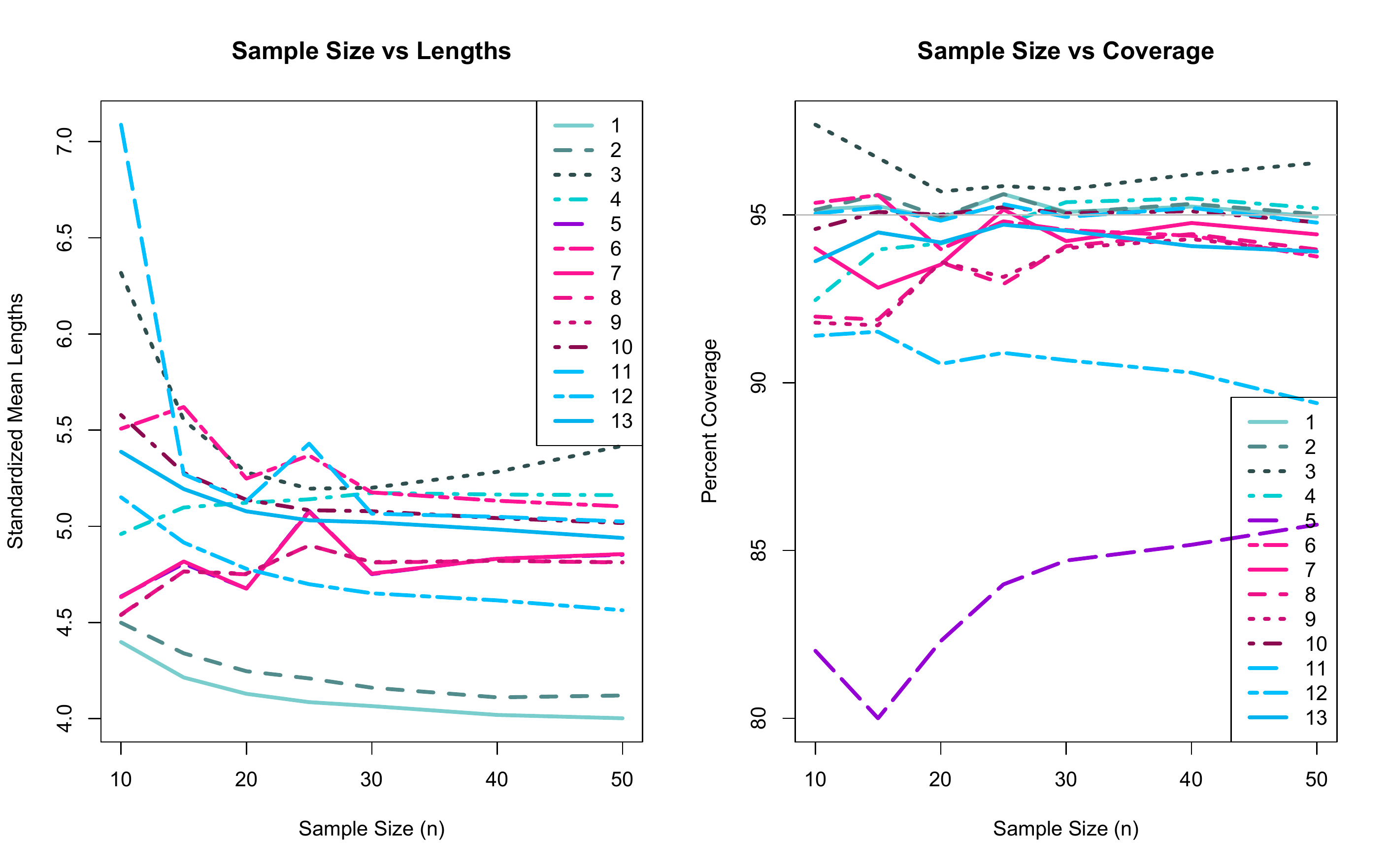} \\ \hline \\
{\bf Distribution:} Cauchy(0,1) \\
\includegraphics[width=\textwidth,height=\myheight,clip=]{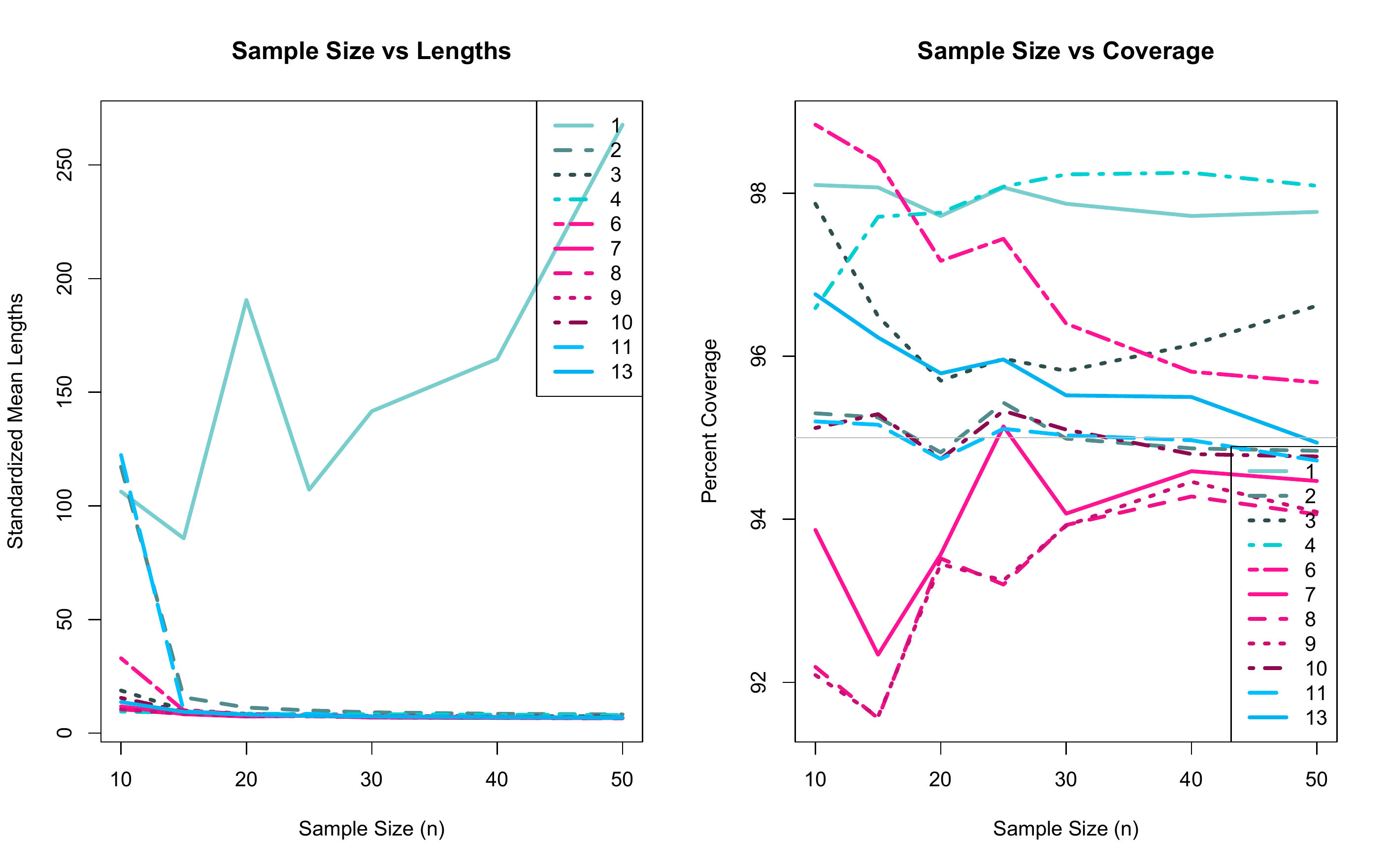} \\ \hline \\
{\bf Distribution:} Uniform[-1,1] \\
\includegraphics[width=\textwidth,height=\myheight,clip=]{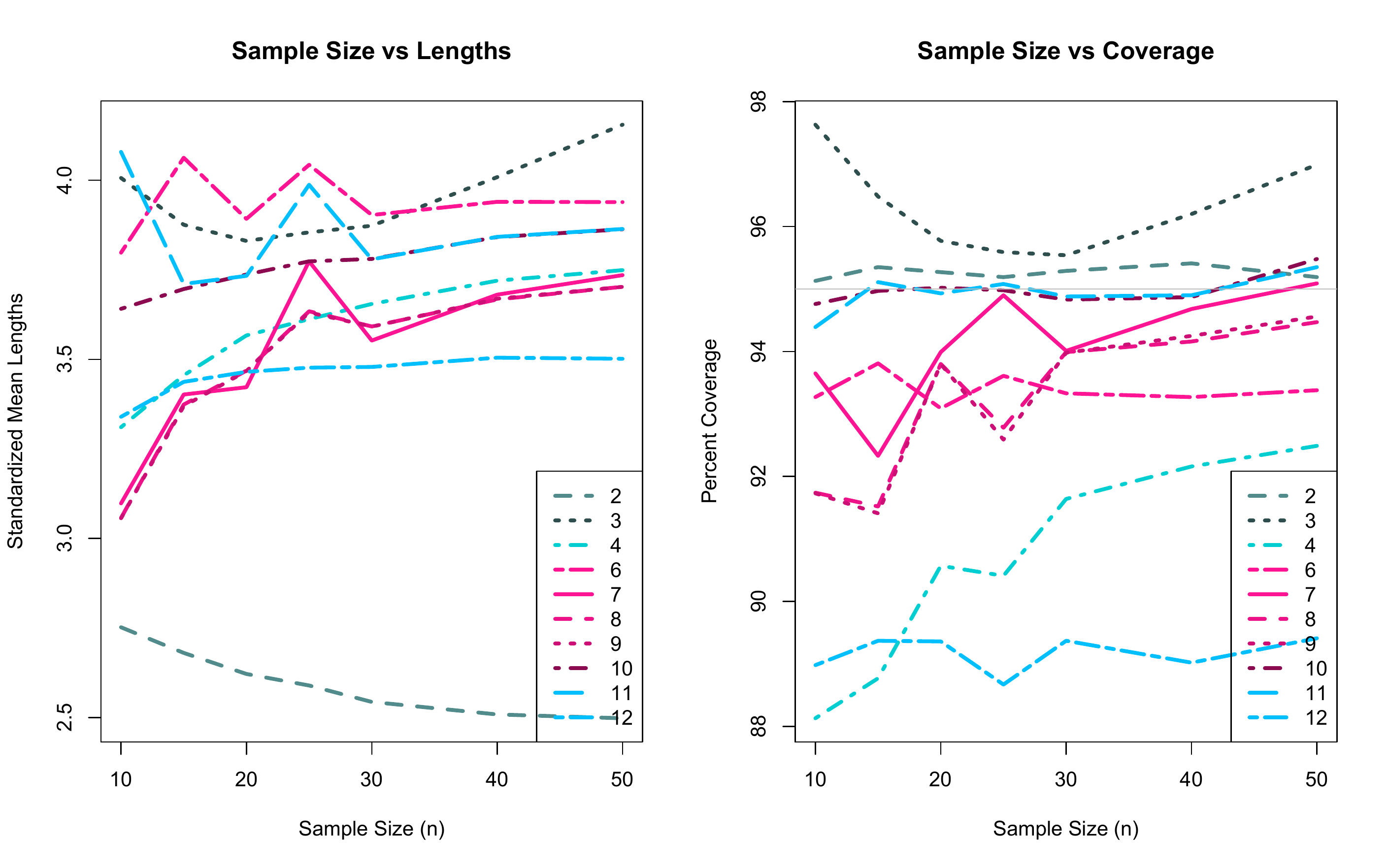} \\ \hline
\end{tabular}
\end{figure}

\begin{figure*}[h]
\caption{Simulation results ... continued}
\label{fig-big-simuls-continued}
\begin{tabular}{c} \hline \\
{\bf Distribution:} Logistic(0,1) \\
\includegraphics[width=\textwidth,height=\myheight,clip=]{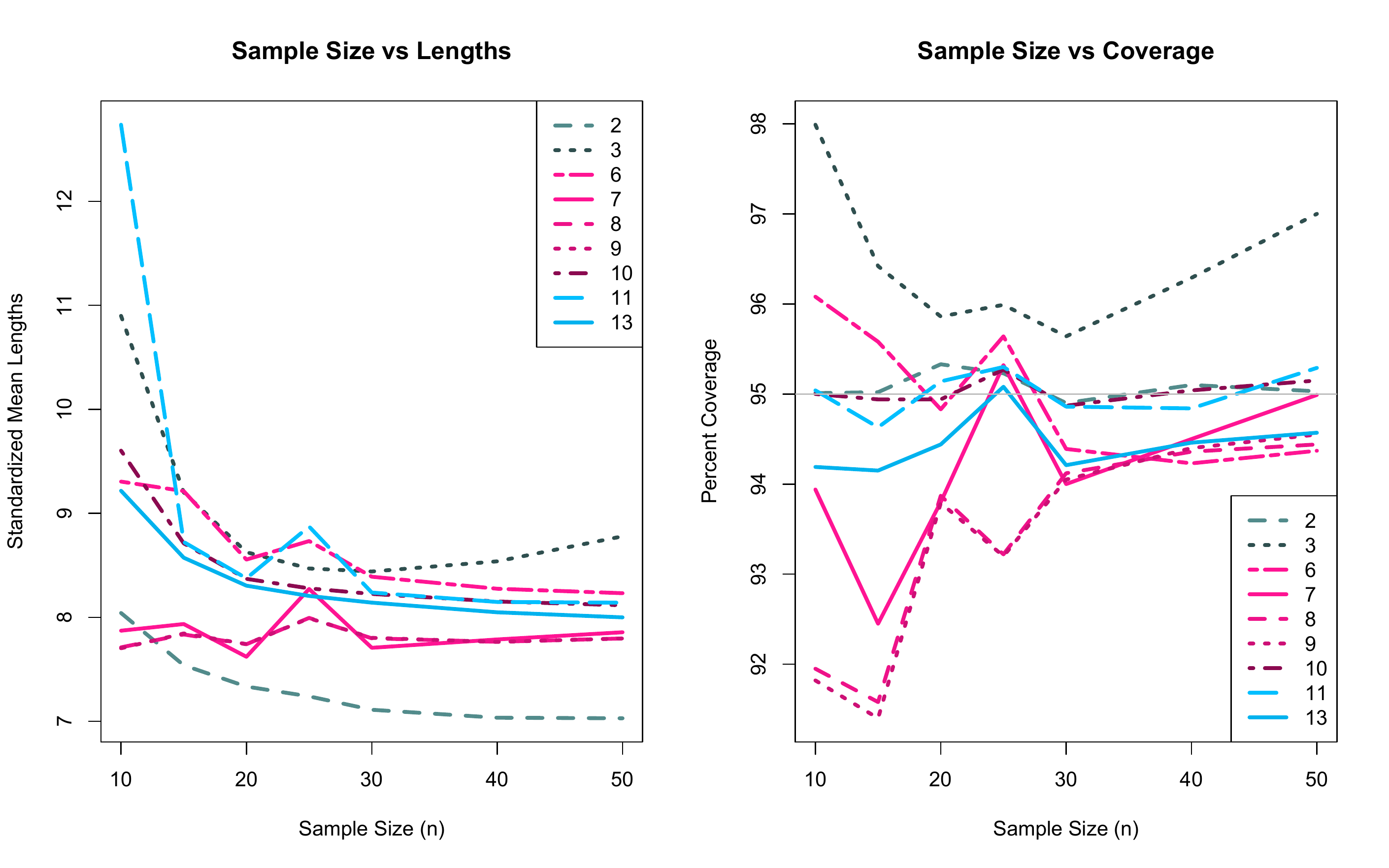} \\ \hline \\
{\bf Distribution:} Gamma(2,1) \\
\includegraphics[width=\textwidth,height=\myheight,clip=]{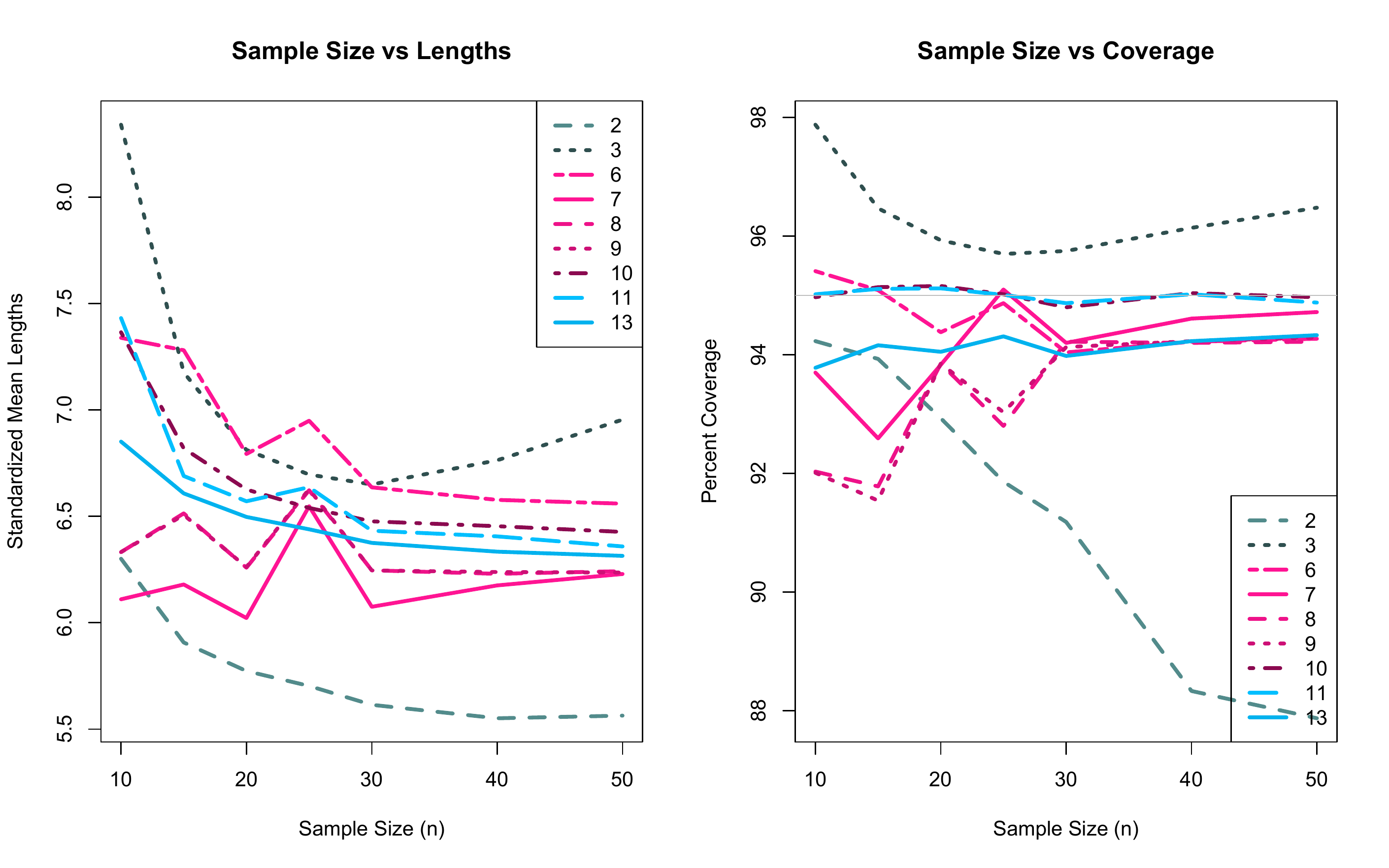} \\ \hline \\
{\bf Distribution:} Weibull(.5,1) \\
\includegraphics[width=\textwidth,height=\myheight,clip=]{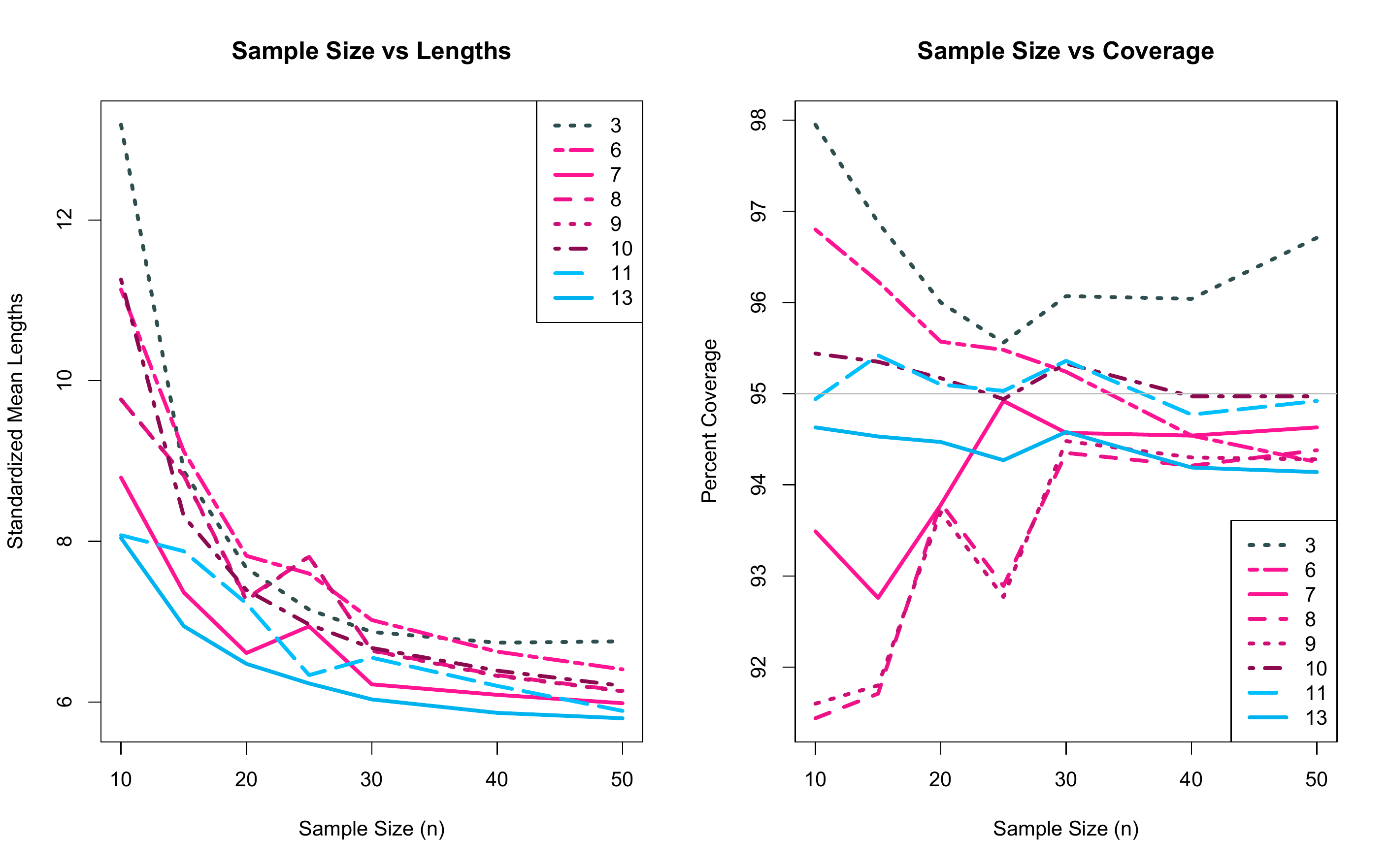} \\ \hline \\
{\bf Distribution:} .6*Normal(-5,3) + .4*Normal(5,2) \\
\includegraphics[width=\textwidth,height=\myheight,clip=]{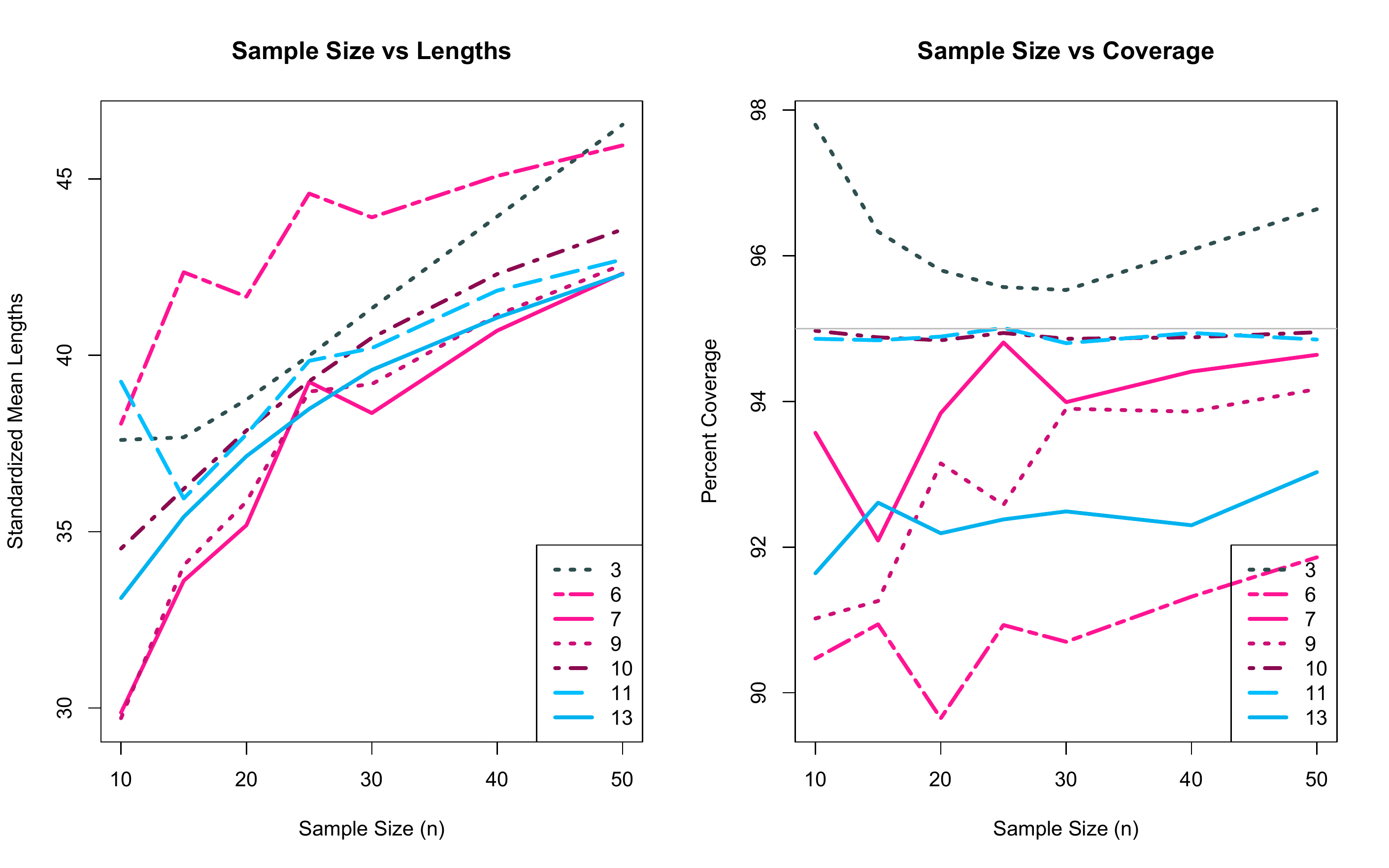} \\ \hline
\end{tabular}
\end{figure*}

\begin{figure}[h]
\caption{Plots of the results for the five final chosen methods with 20000 replications and BREPS=5000 over a set of sample sizes ($n \in \{10,15, 20, 25, 30, 40, 50,75,100\}$) for normal, Cauchy, uniform, logistic, gamma(2,1), Weibull(.5,1), and mixture of normals error distributions. The plots pertaining to the lengths (contents) utilized the {\em standardized mean length}, which is the mean length multiplied by $\sqrt{n}$. The description of these five CR methods are in Table \ref{table-methods list}.}
\label{fig-big-simuls-chosen}
\begin{tabular}{c} \hline \\
{\bf Distribution:} Normal(0,1) \\
\includegraphics[width=\textwidth,height=\myheight,clip=]{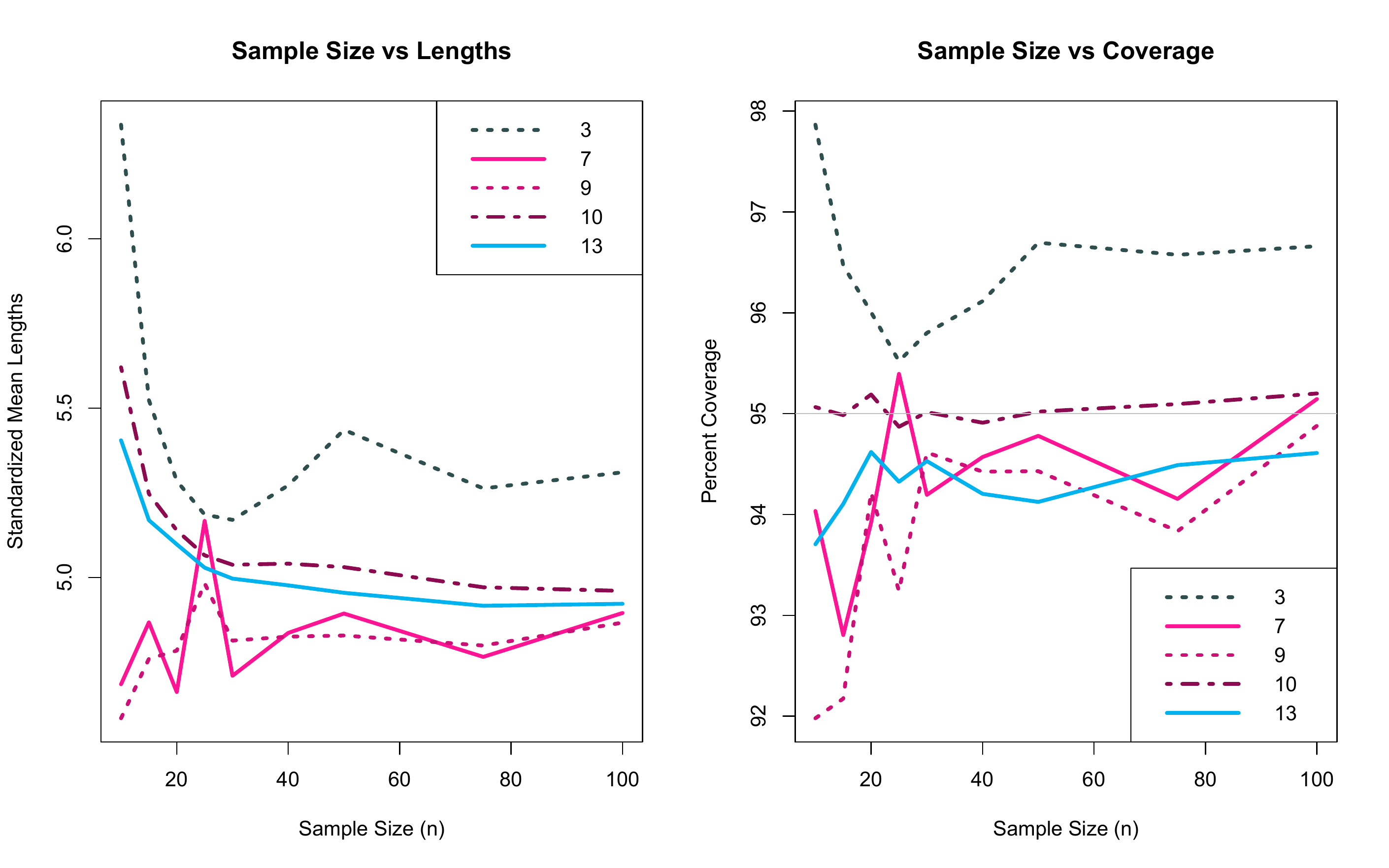} \\ \hline \\
{\bf Distribution:} Cauchy(0,1) \\
\includegraphics[width=\textwidth,height=\myheight,clip=]{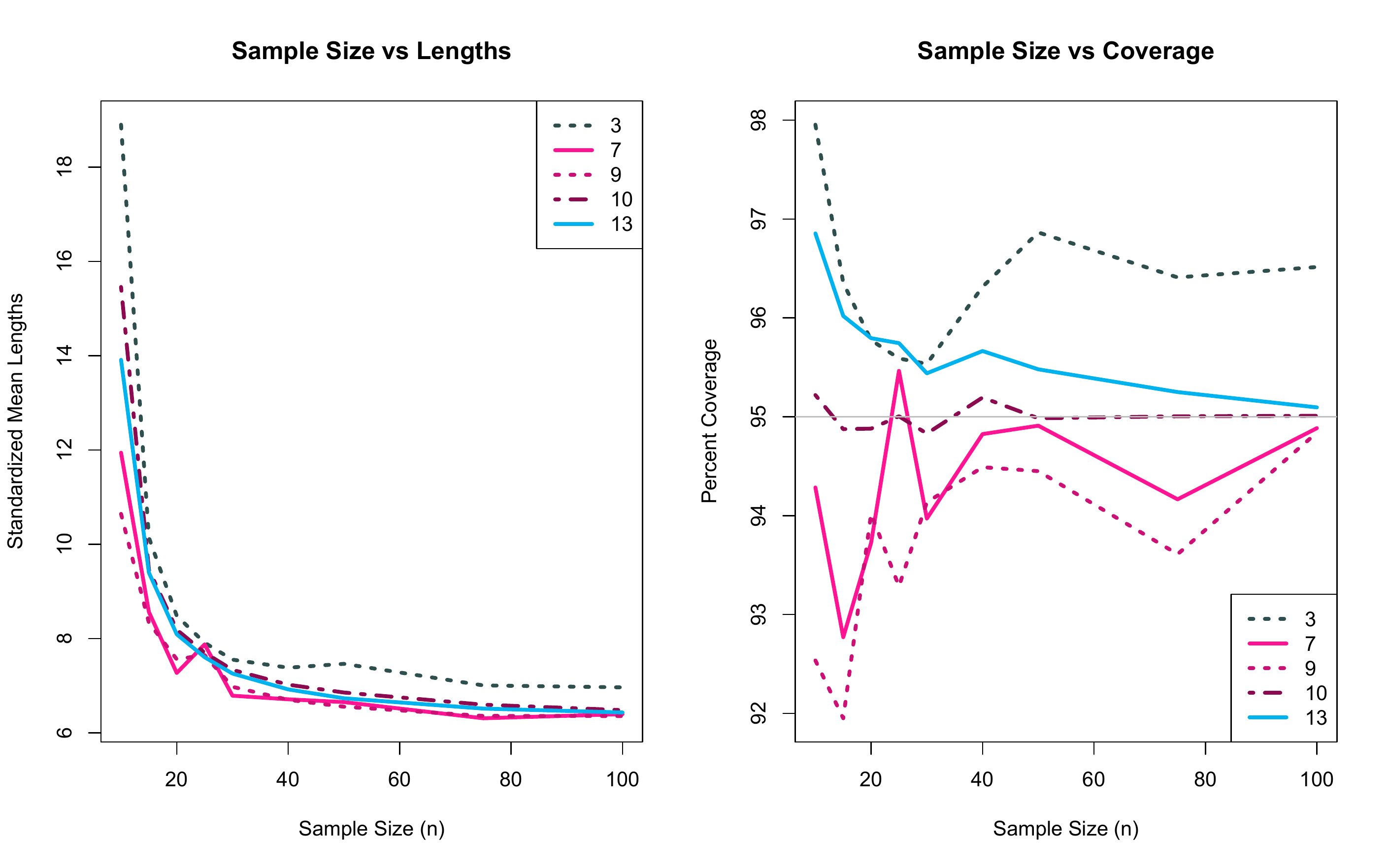} \\ \hline \\
{\bf Distribution:} Uniform[-1,1] \\
\includegraphics[width=\textwidth,height=\myheight,clip=]{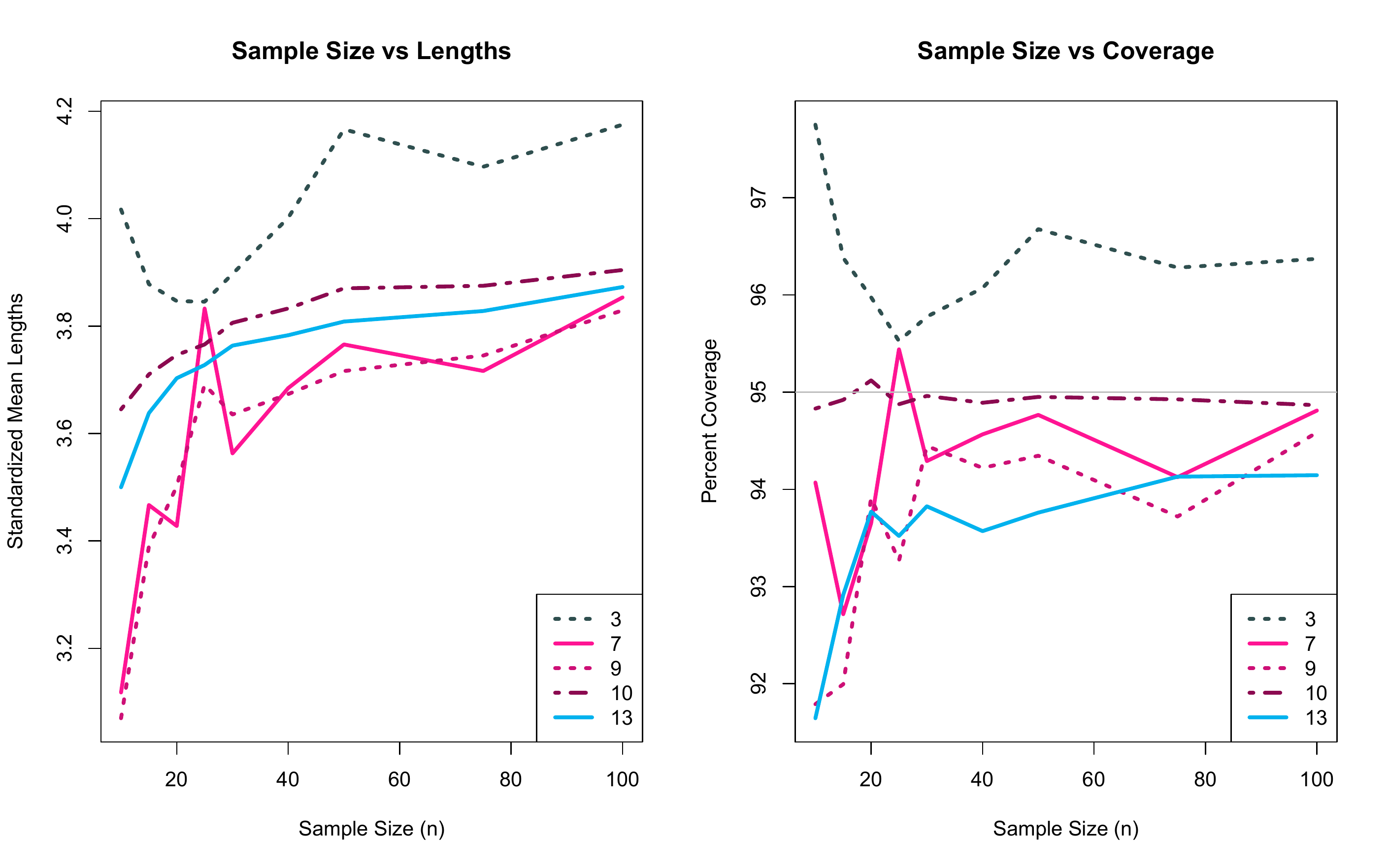} \\ \hline
\end{tabular}
\end{figure}

\begin{figure*}[h]
\caption{Simulation results for chosen methods ... continued}
\label{fig-big-simuls-chosen-continued}
\begin{tabular}{c} \hline \\
{\bf Distribution:} Logistic(0,1) \\
\includegraphics[width=\textwidth,height=\myheight,clip=]{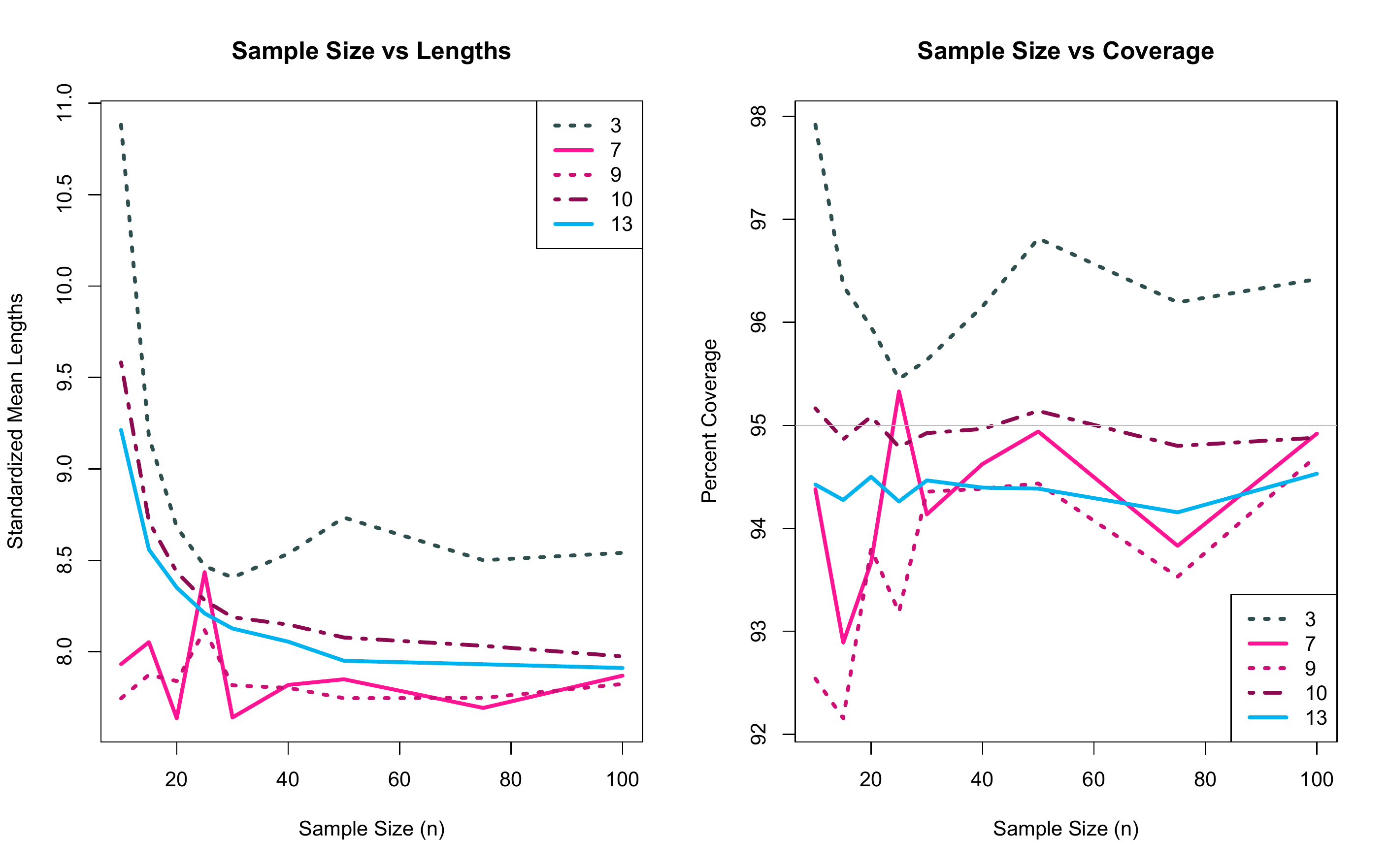} \\ \hline \\
{\bf Distribution:} Gamma(2,1) \\
\includegraphics[width=\textwidth,height=\myheight,clip=]{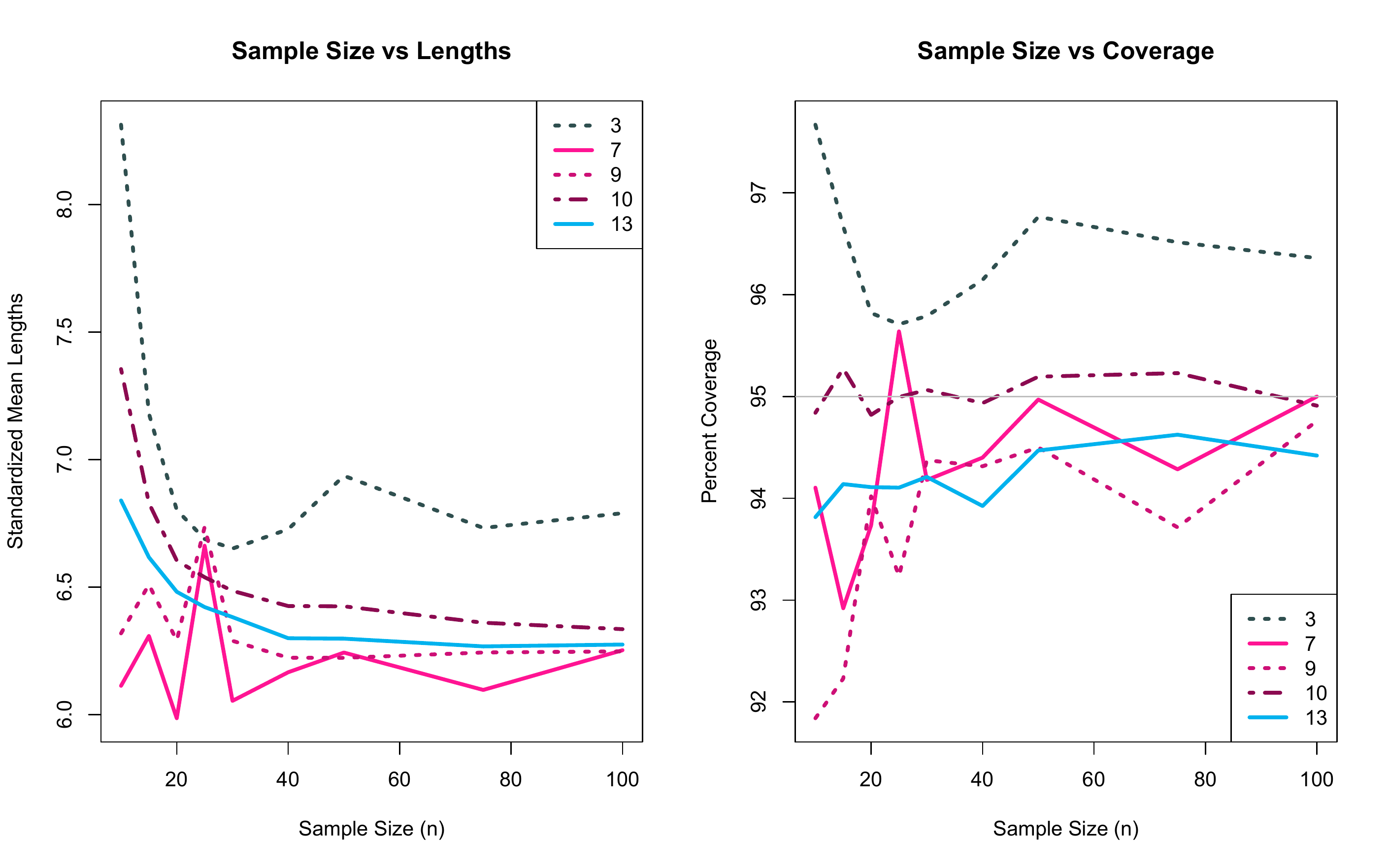} \\ \hline \\
{\bf Distribution:} Weibull(.5,1) \\
\includegraphics[width=\textwidth,height=\myheight,clip=]{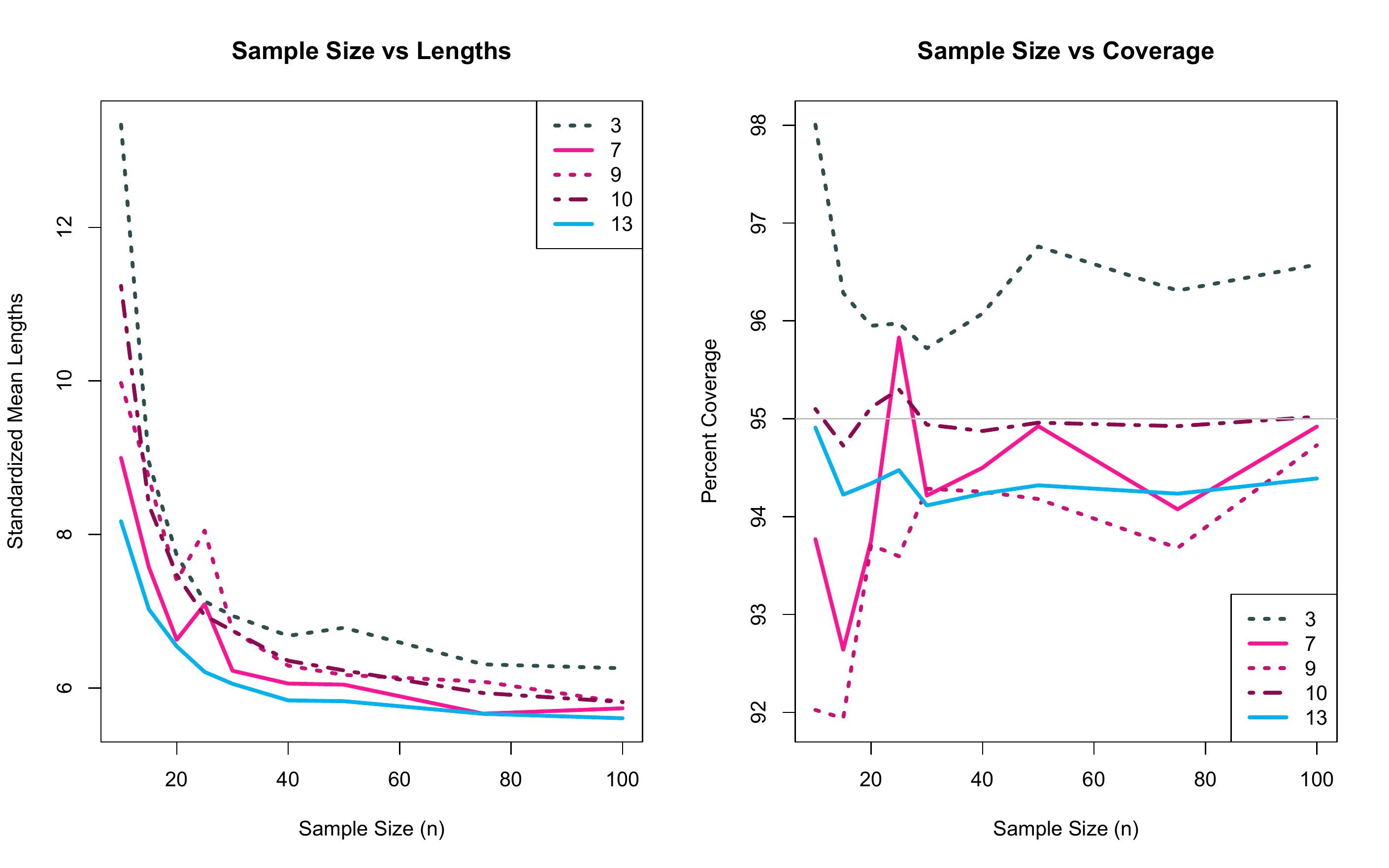} \\ \hline \\
{\bf Distribution:} .6*Normal(-5,3) + .4*Normal(5,2) \\
\includegraphics[width=\textwidth,height=\myheight,clip=]{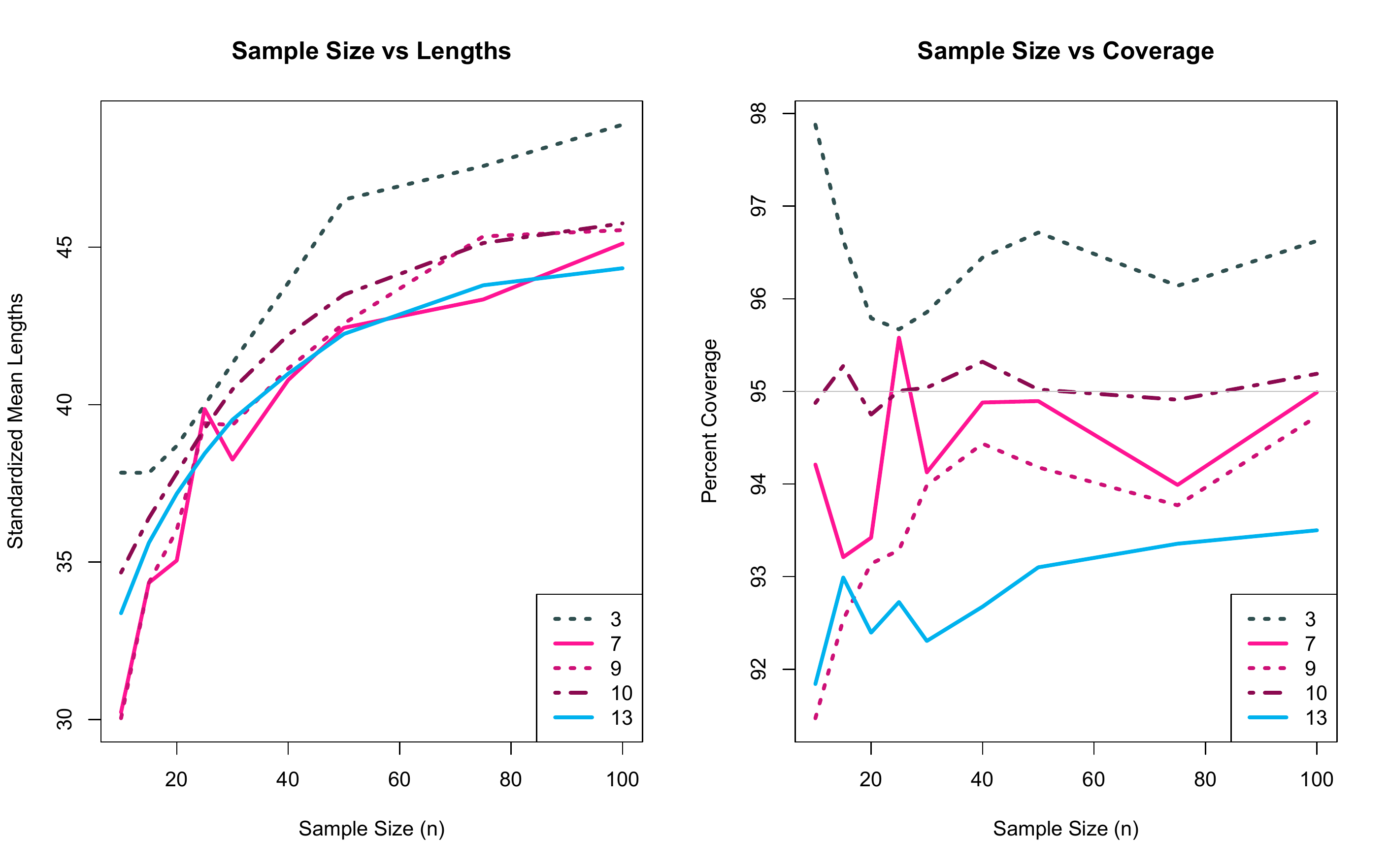} \\ \hline
\end{tabular}
\end{figure*}

\end{document}